\documentclass[11pt]{article}

\usepackage[utf8]{inputenc} 
\usepackage[T1]{fontenc}    
\usepackage{lmodern}

\usepackage[margin=1in]{geometry}
\usepackage[colorlinks,citecolor=blue,urlcolor=blue]{hyperref}      
\usepackage{url}            
\usepackage{multicol,longtable,multirow,booktabs, makecell} 
\usepackage{amsmath, amsfonts, amssymb, amsthm, mathtools,mathrsfs}       
\usepackage{authblk}

\usepackage{nicefrac, eucal, upgreek} 
\usepackage{microtype}      
\usepackage{cleveref}       
\usepackage{graphicx,caption,subcaption}

\usepackage[shortlabels]{enumitem}
\usepackage[authoryear]{natbib}
\theoremstyle{plain}

\newtheorem{proposition}{Proposition}
\newtheorem{theorem}{Theorem}[section]
\newtheorem{lemma}[theorem]{Lemma}

\theoremstyle{remark}
\newtheorem{definition}[theorem]{Definition}

\allowdisplaybreaks

\newcommand{\dd}{\mathop{}\!\mathrm{d}}

\def\E{\mathbb{E}}
\def\R{\mathbb{R}}

\newcommand{\Acal}{\mathcal{A}}

\newcommand{\Mfrak}{\mathfrak{M}}

\newcommand{\dKL}{d_{\mathrm{KL}}}
\newcommand{\dGAB}{d_{\mathrm{GAB}}}
\newcommand{\eGAB}{\mathcal{E}_{\mathrm{GAB}}}

\newcommand{\bb}[1]{\boldsymbol{#1}}

\newcommand{\ind}[1]{\boldsymbol{1}_{#1} }

\newcommand{\norm}[1]{\left\Vert #1 \right\Vert}

\newcommand{\inner}[1]{\left\langle #1\right\rangle}

\newcommand{\tr}{^{\intercal}}

\DeclareMathOperator*{\argmin}{arg\,min}

\title{Characterization of Generalized Alpha-Beta Divergences and Associated Entropy Measures}

\date{}
\author[1]{Subhrajyoty Roy}
\affil[1]{Washington University in St. Louis, St. Louis, USA}

\author[2]{Supratik Basu}
\affil[2]{Duke University, Durham, USA}

\author[3]{Abhik Ghosh}
\author[3]{Ayanendranath Basu}
\affil[3]{Indian Statistical Institute, Kolkata, India}

\begin{document}

\maketitle

\begin{abstract}
Minimum divergence estimators provide a natural framework for robust (parametric) statistical inference. Useful properties of several such divergence measures, including, the Hellinger distance, the power divergence, the density power divergence, the logarithmic density power divergence, etc., have been established in the literature; many of them lead to estimators with high statistical efficiency, sometimes even full asymptotic efficiency. The notable success of these divergences as tools of parametric inference motivates us to explore possible extensions of the alpha-beta divergence family, leading to a superfamily of divergence measures called the ``generalized alpha-beta (GAB) divergences''. This family contains all the aforementioned popular divergence measures as special cases, and additionally provides opportunities to discover new and novel classes of divergences that generate estimators having strong robustness properties without allowing a significant drop in statistical efficiency in various applications. In this paper, we provide the necessary and sufficient conditions for the validity of these generalized divergence measures that enable us to employ them for improved statistical inference. We also show various characterizing properties like duality, inversion, semi-continuity, etc., for the general class of GAB divergences. A discussion on the entropy measure derived from this general family and its properties are also presented  along with the associated maximum entropy principle. The class of GAB divergences provide a delicate balance between local and global robustness, and this is illustrated by two examples of robust parameter estimation under the Geometric and the normal scale models.
\end{abstract}


\section{Introduction}\label{sec:intro}

Information-theoretic divergence measures play a vital role in many areas of statistical inference, supervised and unsupervised machine learning, pattern recognition, signal analysis, and more. Due to their simplicity and natural interpretability, minimum divergence estimators, which minimize the amount of ``separation" between the true data-generating distribution and the model family of distributions, represent a class of popular tools in parametric statistical estimation~\citep{basu2011statistical}. In addition, their strong robustness properties have allowed them to proliferate into modern applications beyond classical statistical inference, including supervised learning~\citep{ghosh2026provably}, boosting algorithms~\citep{takenouchi2004robustifying}, factor analysis~\citep{roy2024robustpca, roy2024robust}, clustering~\citep{chakraborty2023robust}, generative models~\citep{nowozin2016fGAN}, differential privacy~\citep{asoodeh2021local}, reinforcement learning~\citep{li2025choice}, and many more.

Within this realm, certain classes of density-based divergences stand out due to their usefulness in statistical inference, which include the squared Euclidean distance (also referred to as $L^2$ distance) and the Cressie-Read power divergence family~\citep{cressie1984multinomial} covering many notable divergences such as the Kullback-Leibler (KL) divergence, Pearson's chi-square, Neyman's chi-square, Hellinger distance, and several other chi-square type divergences. \cite{basu1998robust} further introduced the density power divergence (DPD) family (also noted as the family of Beta divergences~\citep{cichocki2010families}), which acts as a bridge between the Kullback-Leibler divergence and the $L^2$ distance. Generalizing the above, \cite{ghosh2017generalized} proposed the super divergence (S-divergence) family, which connects the power divergence family with the DPD family. A different generalization starting from the logarithmic density power divergence (LDPD) family of~\cite{jones2001comparison} (also referred to as the family of Gamma divergences by~\cite{fujisawa2008robust}) was proposed by~\cite{maji2014logarithmic}, resulting in the logarithmic super-divergence (LSD) class. Starting with the Alpha divergence~\citep{chernoff1952measure}, \cite{cichocki2011generalized} produced the Alpha-Beta (AB) divergence family, which is closely related to the $S$-divergence family. In another direction, \cite{jones2001comparison} also proposed an alternative divergence family, which provides a smooth bridge between the DPD family and the LDPD family based on a hyperparameter $\phi$, leading to a two-parameter family of $(\phi,\gamma)$-divergences.

Many authors have noted the benefits of robustness by suitably modifying the contribution of the model density -- either through power or logarithmic transformations -- while preserving desired statistical properties; e.g., see~\citep{basu1998robust, basu2011statistical, cichocki2010families, ghosh2013robust,ghosh2017generalized, roy2024robust}. Thus, it is both natural and compelling to investigate whether other useful divergence measures can be systematically generated by applying more general transformations. Such an approach not only provides a unifying perspective on several existing divergence families but also opens up pathways to designing novel divergences with tailored robustness-efficiency trade-offs and desirable analytical properties.

An initial attempt to realize this vision was conceptualized by~\cite[Eq. (1)]{maji2014logarithmic}, but the author did not provide any characterization of the corresponding divergence family. In fact, it was not even known if the form proposed there constitutes a valid statistical divergence, i.e., is nonnegative for all possible pairs of densities. Building on this preliminary insight, our work attempts to understand if the existing choices of divergences can be refined for simultaneous gain in robustness and efficiency, or whether the current formulations already represent a fundamental limit. In particular, following some introductory discussions existing divergences in Section~\ref{sec:prelim} and on generalized alpha-beta (GAB) divergence family in Section~\ref{sec:gsd-alpha-beta}, we establish the necessary and sufficient conditions for the nonnegativity of the GAB divergence for various hyperparameter choices through Theorems~\ref{thm:gab-div-alpha-plus-beta-1}-\ref{thm:gab-div-nec-suff-case4} in Section~\ref{sec:gab-analysis}. In Section~\ref{sec:theory-gab-entropy}, we establish various interesting properties of the associated GAB entropy, including the maximum entropy principle, a much sought property to ensure valid optimization of the relative entropy. Section~\ref{sec:robust-parametric-inference} demonstrates the applicability of these ideas for robust statistical inference, and shows how GAB divergence achieves a balance between local and global robustness through the modification of the curvature of its generating function. We then present a brief illustrative study to highlight these benefits in a discrete setup (geometric distribution family) and in a continuous setup (normal scale estimation problem), with a discussion for future scopes of research following in Section~\ref{sec:conclusion}.

Owing to the breadth of possible statistical consequences and related theoretical questions, it is necessary to approach our problem in a systematic manner.  In order to keep a clear focus in our explorations, the emphasis in this paper is on deriving the necessary and sufficient conditions under which the proposed form of GAB divergences yields a valid statistical divergence. In a companion paper~\citep{roy2026gaboptimal}, we study the asymptotic and robustness properties of the associated minimum divergence estimators, thereby addressing the broader question of identifying optimal choices under the robustness-efficiency trade-off. For brevity of presentation, the technical proofs are deferred till the Appendix.

\section{Preliminaries}\label{sec:prelim}

\subsection{Notation}\label{sec:notation}

We first list all major notations that are used throughout the paper to aid the reader. Let $\Acal$ be the alphabet set, $P$ and $Q$ be two generic sub-probability measures dominated by a common measure $\mu$ on $\Acal$, with sub-density functions $p$ and $q$ (i.e., the Radon-Nikodym derivatives) respectively. The sets of probability measures and the sub-probability measures are denoted by $\Mfrak_{1}(\Acal)$ and $\Mfrak_{\leq 1}(\Acal)$ respectively; note that $\Mfrak_{1}(\Acal)$ is a subset of $\Mfrak_{\leq 1}(\Acal)$. When $\Acal = \{a_1, \dots, a_n\}$ is finite, we use the shorthand $\Mfrak_{1,n}$ and $\Mfrak_{\leq 1, n}$ instead. For $P \in \Mfrak_{\leq 1}(\Acal)$ with sub-density $p$, $\norm{p}_\alpha = (\int_{\Acal} p^\alpha(a)d\mu(a))^{1/\alpha}$ indicates the usual $L_\alpha$-norm, for any $\alpha \neq 0$, provided the integral exists and is finite. If the integral does not converge, we assume $\norm{p}_\alpha = \infty$ by convention. In particular, for finite $\Acal$, this integral always exists and is finite. In the same spirit, we define a $(\alpha,\beta)$-inner product between $P$ and $Q$ with respective sub-densities $p$ and $q$ as
\begin{equation*}
    \inner{p, q}_{\alpha,\beta} = \int_{\Acal} p^\alpha(a) q^\beta(a) \dd\mu(a),
\end{equation*}
\noindent for all $\alpha,\beta \in \R$. Note that this is not a proper inner product in the mathematical sense since it does not satisfy the commutative property. For any $\alpha \neq 0$ and $P \in \Mfrak_{\leq 1}(\Acal)$, let $P^\alpha$ denote the sub-probability distribution corresponding to the sub-density function $p^\alpha$, provided that $\int_{\Acal} p^\alpha(a) d\mu(a) \leq 1$. These are also denoted as the unnormalized $\alpha$-escort distribution of $P$~\citep{cichocki2010families}. It is worthwhile to mention that the bound $1$ is chosen here to ensure that the resulting construction yields a proper sub-probability distribution, but any finite bound works as well after proper scaling without affecting the subsequent discussion. On the other hand, let $P^{[\alpha]}$ and $p^{[\alpha]}$ denote the normalized $\alpha$-escorted versions of the distribution and the density, i.e., $p^{[\alpha]} = p^\alpha/\norm{p}_\alpha^\alpha$~\citep{bercher2010escort, kalogeropoulos2012escort}. For a set $A$ and integer $k \geq 1$, we use the standard notation $C^k(A)$ to denote the class of functions on $A$ which are $k$-times continuously differentiable. Extending this notation, we denote the set of all continuous functions on $A$ as $C^0(A)$. The notation $\ind{A}$ denotes the indicator function of the set $A$, i.e., $\ind{A}(x) = 1$ if $x \in A$ and $0$ otherwise.

\subsection{Existing divergence measures}\label{sec:existing-divergences}

We briefly review a few of the important statistical divergence measures to set up a context for the proposed generalized divergence family. Given $P, Q \in \Mfrak_{\leq 1}(\Acal)$, \cite{cichocki2011generalized} define the Alpha-Beta (AB) divergence as
\begin{equation}
    d_{\mathrm{AB}}^{(\alpha,\beta)}(P, Q)
    = \dfrac{1}{\alpha\beta}\left( \lambda_\alpha \norm{p}_{\alpha+\beta}^{\alpha+\beta} + \lambda_\beta \norm{q}_{\alpha+\beta}^{\alpha+\beta} - \inner{p,q}_{\alpha,\beta} \right)
    \label{eqn:defn-ab-div}
\end{equation}
\noindent where $\lambda_\alpha = \alpha/(\alpha+\beta)$ and $\lambda_\beta = (1 - \lambda_\alpha)$, if $\alpha\beta(\alpha+\beta)\neq 0$. However, if $\alpha\beta(\alpha+\beta) = 0$, the form of AB divergence is defined as the appropriate limits of the form given in~\eqref{eqn:defn-ab-div}. With the restriction when both $\alpha$ and $\beta$ are nonzero, and $(\alpha + \beta) = (1 + \tau)$ for some $\tau \geq 0$, a simple reorganization of the terms of~\eqref{eqn:defn-ab-div} makes the AB divergence equivalent to the S-divergence introduced by~\cite{ghosh2017generalized}, because of the relation,
\begin{equation*}
    d_{\mathrm{S}}^{(\tau,\lambda)}(P, Q)
    = \frac{1}{\tilde{\alpha}} \norm{f}_{1+\tau}^{1+\tau} - \frac{1+\tau}{\tilde{\alpha}\tilde{\beta}}\inner{f,g}_{\tilde{\beta}, \tilde{\alpha}} + \frac{1}{\tilde{\beta}}\norm{g}_{1+\tau}^{1+\tau} = (1+\tau)d_{AB}^{(\tilde{\beta}, \tilde{\alpha})}(P, Q),
\end{equation*}
\noindent where $\tilde{\alpha} = 1 + \lambda(1-\tau)$ and $\tilde{\beta} = \tau - \lambda(1-\tau)$.
\noindent In another direction, taking $\beta = (1 - \alpha)$ reduces the AB divergence to the Alpha divergence family~\citep{chernoff1952measure} (also known as the Cressie Read power divergence family~\citep{cressie1984multinomial}) given by,
\begin{equation}
    d_{\mathrm{A}}^{(\alpha)}(P, Q)
    = \begin{cases}
        \frac{1}{1-\alpha} \norm{p}_1 + \frac{1}{\alpha}\norm{q}_1 - \frac{1}{\alpha(1-\alpha)} \inner{p,q}_{\alpha,1-\alpha}, & \text{if } \alpha \notin \{ 0, 1 \}, \\
        \int p\ln(p/q) - \norm{p}_1 + \norm{q}_1,                                                                              & \text{if } \alpha = 1,               \\
        \int q\ln(q/p) - \norm{q}_1 + \norm{p}_1,                                                                              & \text{if } \alpha = 0,
    \end{cases}
    \label{eqn:defn-a-div}
\end{equation}
\noindent which is also a special case of the S-divergence with $\tau = 0$. On the other hand, when $\alpha = 1$, the alpha-beta divergence reduces to the Beta divergence, defined as
\begin{equation}
    d_{\mathrm{B}}^{(\beta)}(P, Q)
    = \begin{cases}
        \frac{1}{\beta(1+\beta)}\norm{p}_{1+\beta}^{1+\beta} - \frac{1}{\beta}\inner{p, q}_{1,\beta} + \frac{1}{1+\beta}\norm{q}_{1+\beta}^{1+\beta}, & \text{if } \beta \notin \{ -1, 0 \} \\
        \int p\ln(p/q) - \norm{p}_1 + \norm{q}_1,                                                                                                     & \text{if } \beta = 0,               \\
        \int \ln(q/p) + \int p/q - 1,                                                                                                                 & \text{if } \beta = -1,
    \end{cases}
    \label{eqn:defn-b-div}
\end{equation}
\noindent The class of Beta divergence contains a scaled version of DPD family~\citep{basu1998robust} as a special case with $\beta > 0$. We present the above definitions to include sub-probability distributions; the restrictions to probability distributions lead to further simplifications of the aforementioned forms.

Generalizing the Alpha-Beta (AB) divergence with a logarithmic transformation, \cite{cichocki2011generalized} indicated a new family of divergences
\begin{equation}
    d_{\mathrm{AC}}^{(\alpha,\beta)}(P, Q)
    = \frac{1}{\beta(\alpha+\beta)} \ln\left( \norm{p}_{\alpha+\beta}^{\alpha+\beta} \right)
    + \frac{1}{\alpha(\alpha+\beta)} \ln\left( \norm{q}_{\alpha+\beta}^{\alpha+\beta} \right)
    - \frac{1}{\alpha\beta} \ln\left( \inner{p,q}_{\alpha,\beta} \right),
    \label{eqn:defn-ac-div}
\end{equation}
\noindent for $\alpha \beta (\alpha+\beta) \neq 0$. With $\alpha = 1$ and $\beta = (\gamma -1)$, it leads to the Gamma divergence family~\citep{cichocki2010families}. In a separate attempt, \cite{maji2014logarithmic,maji2016logarithmic} also considered the same logarithmic transformation starting with the S-divergence (SD) family, which resulted in the LSD family.

Additional transformations beyond logarithms have also been considered in the literature. \cite{jones2001comparison} proposed a class of divergences which we refer to as the J2 divergence, and it is given by
\begin{equation*}
    d_{\mathrm{J2}}^{(\phi, \gamma)}(P, Q) = \frac{1}{\phi}\norm{p}_{1+\gamma}^{(1+\gamma)\phi} + \frac{1}{\gamma\phi}\norm{q}_{1+\gamma}^{(1+\gamma)\phi} - \frac{1+\gamma}{\gamma\phi} \inner{p,q}_{\gamma,1}^\phi,
\end{equation*}
\noindent for $\gamma \in [0, 1]$ and $\phi > 0$. \cite{kuchibhotla2019statistical} and later~\cite{gayen2024unified} suggested the bridge divergence family to avoid scaling issues with logarithms, resulting in a divergence of the form
\begin{equation*}
    d_{\mathrm{Bridge}}^{(\alpha)}(P, Q) = \ln(\lambda + \bar{\lambda} \norm{p}_{1+\alpha}^{1+\alpha}) + \frac{1}{\alpha} \ln(\lambda + \bar{\lambda} \norm{q}_{1+\alpha}^{1+\alpha}) - \frac{1+\alpha}{\alpha} \ln(\lambda + \bar{\lambda} \inner{p,q}_{\alpha,1}),
\end{equation*}
\noindent for some $\lambda = 1 - \bar{\lambda} \in [0, 1]$. All these transformations motivate the development of Generalized Alpha-Beta divergence, as a means of controlling the geometry of the relevant dual spaces~\citep{amari2010information}, which is discussed in the following section.

\section{Generalized Alpha-Beta divergence \& Entropy}\label{sec:gsd-alpha-beta}

When $\alpha, \beta, (\alpha+\beta) \neq 0$, aforementioned statistical divergence measures discussed in Section~\ref{sec:existing-divergences} have a similar form, two terms containing the $(\alpha+\beta)$-th norm of the densities $p$ and $q$, and a single term containing the $(\alpha,\beta)$-inner product between $p$ and $q$. Motivated by such a structure and their logarithmic transform, \cite{maji2019contributions} conceptualized a generalized class of divergences considering general transformations beyond logarithm, but without any discussion on the relevant characterization properties. We formalize this general divergence family below, allowing $\alpha$ and $\beta$ to take any real value.

\subsection{Definition of GAB divergences family}\label{sec:gab-defn}

Consider two hyperparameters $\alpha, \beta \in \R$ and let $\psi: [0, \infty] \rightarrow \R$ be a suitable transformation. We define the Generalized Alpha-Beta (GAB) divergence between any two sub-probability distributions $P$ and $Q$ as
\begin{equation}
    \dGAB^{(\alpha,\beta),\psi}(P, Q)
    = \dfrac{1}{\beta(\alpha+\beta)} \psi\left( \norm{p}_{\alpha+\beta}^{\alpha+\beta} \right)
    + \dfrac{1}{\alpha(\alpha+\beta)} \psi\left( \norm{q}_{\alpha+\beta}^{\alpha+\beta} \right)
    - \dfrac{1}{\alpha\beta} \psi\left( \inner{p,q}_{\alpha,\beta} \right),
    \label{eqn:defn-gab-div}
\end{equation}
\noindent when $\alpha, \beta, (\alpha+\beta)$ are all nonzero. If $\psi \in C^1([0,\infty))$, then one can define the form of Generalized Alpha-Beta (GAB) divergence for the edge cases, i.e., when $\alpha = 0, \beta = 0$ or $\alpha + \beta = 0$, by taking the corresponding limits of the GAB divergence, leading to the forms:
\begin{equation}
    \dGAB^{(\alpha,\beta),\psi}(P, Q)
    = \begin{cases}
        \alpha^{-2}\left[ \psi'(\norm{p}_{\alpha}^{\alpha}) \dKL(P^\alpha,Q^\alpha) - \psi(\norm{p}_\alpha^\alpha) + \psi(\norm{q}_\alpha^\alpha)\right], & \text{if } \alpha \neq 0, \beta = 0, \\
        \beta^{-2} \left[\psi'(\norm{q}_{\beta}^{\beta}) \dKL(Q^\beta, P^\beta) - \psi(\norm{q}_\beta^\beta) + \psi(\norm{p}_\beta^\beta)\right],         & \text{if } \alpha = 0, \beta \neq 0, \\
        \alpha^{-2}\left[ \psi'(1) \int \ln(q^\alpha/p^\alpha) + \psi\left( \norm{p/q}_\alpha^{\alpha} \right) - \psi(1) \right],                         & \text{if } \alpha = -\beta \neq 0,   \\
        \frac{\psi'(1)}{2}\int (\ln p - \ln q)^2,                                                                                                         & \text{if } \alpha = \beta = 0.
    \end{cases}
    \label{eqn:defn-gab-div-edge}
\end{equation}
\noindent where
\begin{equation}
    \dKL(P, Q) = \int p\ln \left( p/q \right) d\mu,
    \label{eqn:d-kl-1}
\end{equation}
\noindent denotes the Kullback-Leibler (KL) divergence between any two nonnegative measures $P$ and $Q$ with densities $p$ and $q$ respectively. In Eq.~\eqref{eqn:defn-gab-div-edge}, the quantity $\dKL(P^\alpha, Q^\alpha)$ thus indicates the KL divergence between the unnormalized $\alpha$-escorted versions of $P$ and $Q$. It is quite clear that different choices of $\psi(x)$ lead to the special cases described in Section~\ref{sec:existing-divergences}, e.g., the identity function leads to the AB divergence, $\psi(x) = \ln(x)$ leads to the logarithmic AB divergence, power transformation leads to the J2 divergence, etc. Table~\ref{tab:gab-divergence-families} demonstrates these connections in an explicit manner. In general, some assumptions are necessary for this characterizing transformation $\psi$ to produce a well-defined statistical divergence measure, and these conditions are described later in detail in Section~\ref{sec:gab-char}.

\begin{table*}[!t]
    \caption{Different families of statistical divergence that arise as special cases of the Generalized Alpha-Beta divergence family, up to a multiplicative constant. (~$^\ast$An unscaled version of $\gamma$-divergence was developed by~\cite{jones2001comparison} and is called the logarithmic density power divergence.)}
    \label{tab:gab-divergence-families}
    \centering
    \resizebox{\linewidth}{!}{\begin{tabular}{llll}
            \toprule
            \textbf{Divergence Family}                                                                & \textbf{Form of the divergence}                                                                   & \textbf{$\psi(x)$}                                            & \textbf{Hyperparameters}                                      \\
            \midrule
            \makecell[l]{Alpha-Beta divergence \\ \citep{cichocki2011generalized}}                                                               &
            $\begin{aligned}
                     & (\beta(\alpha+\beta))^{-1} \norm{p}^{\alpha+\beta}_{\alpha+\beta} - (\alpha\beta)^{-1}\inner{p, q}_{\alpha,\beta} \\[-0.2ex]
                     & \qquad \qquad \qquad \qquad + (\alpha(\alpha+\beta))^{-1} \norm{q}_{\alpha+\beta}^{\alpha+\beta}
                \end{aligned}$ & $x$ & $\alpha, \beta \in \R \setminus \{0\}, (\alpha +\beta)\neq 0$                                                                 \\
            \makecell[l]{Density Power divergence \\ \citep{basu1998robust}}                                                                 & $\norm{p}^{1+\alpha}_{1+\alpha} - \left(1 + \alpha^{-1} \right)\inner{p,q}_{\alpha,1} + \alpha^{-1} \norm{q}^{1+\alpha}_{1+\alpha}$                                             & $x$                                                           & $\alpha \geq 0, \beta = 1$                                    \\
            \makecell[l]{Power divergence \\ \citep{cressie1984multinomial}}                                                                 & $[\lambda(\lambda - 1)]^{-1}\left[ \int p^\lambda q^{1-\lambda} - \lambda \int p + (\lambda - 1)\int q \right]$                                                                 & $x$                                                           & $\alpha = \lambda, \beta = (1-\lambda)$                       \\
            \makecell[l]{S-divergence \\ \citep{ghosh2017generalized}}                                                                       & $\frac{1}{A}\norm{p}_{1+a}^{1+a} - \frac{1+a}{AB}\inner{p,q}_{B,A} + \frac{1}{B} \norm{q}_{1+a}^{1+a}$                                                                          & $x$                                                           & $\alpha = B, \beta = A = (1+a - \alpha)$                      \\
            \midrule
            \makecell[l]{AC divergence \\ \citep{cichocki2010families}}                                                                      & $\ln\left[ \norm{p}_{\alpha+\beta}^{\beta} \norm{q}_{\alpha+\beta}^\alpha \right] - \ln\left[\inner{p, q}_{\alpha,\beta}^{1/\alpha\beta} \right]$                               & $\ln(x)$                                                      & $\alpha, \beta \in \R \setminus \{0\}, (\alpha +\beta)\neq 0$ \\
            \makecell[l]{Logarithmic S-divergence \\ \citep{maji2014logarithmic}}                                                            & $\ln\left[ \norm{p}_{1+a}^{(1+a)/A} \norm{q}_{1+a}^{(1+a)/B} \right] - \ln\left[\inner{p, q}_{B,A}^{(1+a)/AB} \right]$                                                          & $\ln(x)$                                                      & $\alpha = B, \beta = A = (1+a - \alpha)$                      \\
            \makecell[l]{$\gamma$-divergence$^\ast$ \\ \citep{cichocki2010families}}                                                                & $\frac{1}{\gamma(\gamma - 1)}\ln\left[ \dfrac{\int p^\gamma q^{1-\gamma}}{(\int p)^\gamma (\int q)^{1 - \gamma}} \right]$                                                       & $\ln(x)$                                                      & $\alpha = \gamma, \beta = (1-\gamma)$                         \\
            \midrule
            \makecell[l]{J2 divergence \\ \citep{jones2001comparison}}                                                                     & $\phi^{-1}\left( \int p^{1+\gamma} \right)^\phi - \frac{1+\gamma}{\gamma \phi}\left( \int p^\gamma q \right)^\phi + \frac{1}{\gamma \phi}\left( \int q^{1+\gamma} \right)^\phi$ & $\phi^{-1} x^\phi$                                            & $\alpha = \gamma, \beta = 1$                                  \\
            \midrule                                         \\[-2ex]
            \makecell[l]{Bridge divergence \\ \citep{kuchibhotla2019statistical}}                                                            &
            $\begin{aligned}
                     & \ln(\lambda + \bar{\lambda} \norm{p}_{1+\alpha}^{1+\alpha})                              \\[-0.2ex]
                     & \qquad \qquad + \alpha^{-1} \ln(\lambda + \bar{\lambda} \norm{q}_{1+\alpha}^{1+\alpha})    \\[-0.2ex]
                     & \qquad \qquad - (1+\alpha)\alpha^{-1} \ln(\lambda + \bar{\lambda} \inner{p, q}_{\alpha,1}) \\[0.5ex]
                \end{aligned}$ & $\ln(\lambda + \bar{\lambda}x)$                                                                                                                                                 & $\alpha \geq 0, \beta = 1, \bar{\lambda} = 1 - \lambda$                                                                                                                                                                                     \\
            \bottomrule
        \end{tabular}}
\end{table*}

\subsection{Connection to other generalization attempts}\label{sec:other-generalization}

Before proceeding further to analyze the GAB divergence form given in Eq.~\eqref{eqn:defn-gab-div}-\eqref{eqn:defn-gab-div-edge}, we present some insights and intuitive understanding of the proposed divergence. Consider the special case when exactly one of $\alpha$ and $\beta$ is equal to $0$. Noting the following equality
\begin{equation*}
    \dKL(P^\alpha, Q^\alpha) = \norm{p}_\alpha^\alpha \left[ \dKL(P^{[\alpha]}, Q^{[\alpha]}) + \ln(\norm{p}_\alpha^\alpha) - \ln(\norm{q}_\alpha^\alpha) \right],
\end{equation*}
\noindent we can rewrite the GAB divergence for $\alpha \neq 0, \beta = 0$ case as
\begin{multline}
    \alpha^2 \dGAB^{(\alpha,0),\psi}(P, Q)
    = \psi'(\norm{p}_\alpha^\alpha) \norm{p}_\alpha^\alpha \left( \dKL(P^{[\alpha]}, Q^{[\alpha]}) + \ln(\norm{p}_\alpha^\alpha) - \ln(\norm{q}_\alpha^\alpha) \right)\\
    - \psi(\norm{p}_\alpha^\alpha) + \psi(\norm{q}_\alpha^\alpha).
    \label{eqn:gab-to-kl-bzero}
\end{multline}
\noindent Similarly, for $\alpha = 0, \beta \neq 0$ case, we get
\begin{multline}
    \beta^2 \dGAB^{(0,\beta),\psi}(P, Q)
    = \psi'(\norm{q}_\beta^\beta) \norm{q}_\beta^\beta \left( \dKL(Q^{[\beta]}, P^{[\beta]}) + \ln(\norm{q}_\beta^\beta) - \ln(\norm{p}_\beta^\beta) \right)\\
    - \psi(\norm{q}_\beta^\beta) + \psi(\norm{p}_\beta^\beta).
    \label{eqn:gab-to-kl-azero}
\end{multline}
\noindent This connection demonstrates that the GAB divergence for $\alpha \neq 0, \beta = 0$ or $\alpha = 0, \beta \neq 0$ can be visualized as a localized affine transformation of the KL divergence between the $\alpha$-escorted versions of the corresponding densities. Focusing on the case when $\alpha \neq 0, \beta = 0$, one can view GAB divergence in a more general form
\begin{equation*}
    \psi(\norm{q}_{\alpha}^{\alpha}) - \psi(\norm{p}_{\alpha}^{\alpha}) + \psi'(\norm{p}_\alpha^\alpha) d^\ast(P^\alpha, Q^\alpha),
\end{equation*}
\noindent where $d^\ast(P^\alpha, Q^\alpha)$ is a discrepancy measure with the property that $d^\ast(P^\alpha, P^\alpha) = 0$. This general form also unifies the existing generalization attempts such as the unified Bregman representation of~\citep{roy2019density} which uses $d^\ast(P, Q) = \int (p - q)d\mu$ and the Norm-based Bregman density power divergence of~\citep{kobayashi2025unified} which arises from $d^\ast(P, Q) = \int (p/q - 1)p^\gamma d\mu$.

\subsection{Generalized Alpha-Beta Entropy}\label{sec:gab-entropy}

An important concept related to a statistical divergence measure is its associated entropy. Entropy measures the amount of ``surprise'' present in a probability distribution $P$, serving as a fundamental metric for uncertainty. Beyond its foundational role in source coding and optimal code-length determination within information theory, the concept is central to characterizing macroscopic disorder in statistical thermodynamics, quantifying predictive uncertainty in machine learning, tracking structural degradation in software engineering, and evaluating predictability within dynamical systems and ergodic theory, etc. Starting from a statistical divergence measure, the associated entropy can be obtained by rewriting the negative divergence of any probability measure $P$ to a uniform measure over the alphabet set $\mathcal{A}$, denoted as $Q_0$~\citep{cichocki2010families}. Since such a uniform measure may not exist when the alphabet set $\Acal$ is infinite, for all discussions related to the entropy, we restrict our attention to a finite $\Acal$ with $n$ elements. Note that, for the proposed GAB divergence, the second term in $\dGAB^{(\alpha,\beta),\psi}(P, Q_0)$ is free of $P$. As a result, we may define the corresponding entropy measure as
\begin{equation}
    \eGAB^{(\alpha,\beta), \psi}(P) := -\dfrac{1}{\beta} \left[ \frac{\psi\left( \norm{p}_{\alpha+\beta}^{\alpha+\beta} \right)}{\alpha+\beta} - \frac{\psi\left( \norm{p}_{\alpha}^{\alpha} \right)}{\alpha} \right],
    \label{eqn:gab-entropy}
\end{equation}
\noindent for $\alpha, \beta, (\alpha+\beta) \neq 0$. If $\psi(x) = \ln(x)$, then Eq.~\eqref{eqn:gab-entropy} reduces to the logarithmic $(\alpha,\alpha+\beta)$-norm entropy $\mathcal{E}_{LN}$ up to a multiplicative constant as expressed in the following relation
\begin{equation*}
    \eGAB^{(\alpha,\beta), \ln}(P)
    = \dfrac{1}{\alpha - (\alpha+\beta)}\left[ \ln\left( \norm{p}_{\alpha+\beta} \right) - \ln\left( \norm{p}_{\alpha} \right) \right]
    = \dfrac{1}{\alpha(\alpha+\beta)}\mathcal{E}_{LN}^{(\alpha,\alpha+\beta)}(P).
\end{equation*}
\noindent The logarithmic norm entropy is a generalization of the class of R\'{e}nyi entropy (hence also the Shannon entropy) that satisfies an important scale-invariance property~\citep{ghosh2021scale}. Note that, the Generalized Alpha-Beta Entropy (GABE) defined in Eq.~\eqref{eqn:gab-entropy} can be trivially extended to encompass the class of sub-probability distributions $\Mfrak_{\leq 1, n}$ as well.

For the edge cases when either $\alpha, \beta$ or $(\alpha+\beta)$ is equal to $0$, the form of GABE reduces to
\begin{equation*}
    \eGAB^{(\alpha,\beta), \psi}(P)
    = \begin{cases}
        \alpha^{-2}\psi\left( \norm{p}_\alpha^\alpha \right) - \alpha^{-1}\psi'\left( \norm{p}_\alpha^\alpha \right)\int p^\alpha \ln(p), & \text{if } \alpha \neq 0, \beta = 0, \\
        \beta^{-1} \psi'(1) \int \ln(p) - \beta^{-2}\psi\left( \norm{p}_\beta^\beta \right),                                              & \text{if } \alpha = 0, \beta \neq 0, \\
        \alpha^{-1}\psi'(1) \int \ln(p) - \alpha^{-2}\psi\left( \norm{p}_\alpha^\alpha \right),                                           & \text{if } \alpha = -\beta \neq 0,   \\
        -\frac{\psi'(1)}{2} \int (\ln(p))^2,                                                                                              & \text{if } \alpha = \beta = 0,
    \end{cases}
\end{equation*}
\noindent which can be easily verified by taking appropriate limits in Eq.~\eqref{eqn:gab-entropy}.

\begin{figure*}[t]
    \centering
    \includegraphics[width=\linewidth]{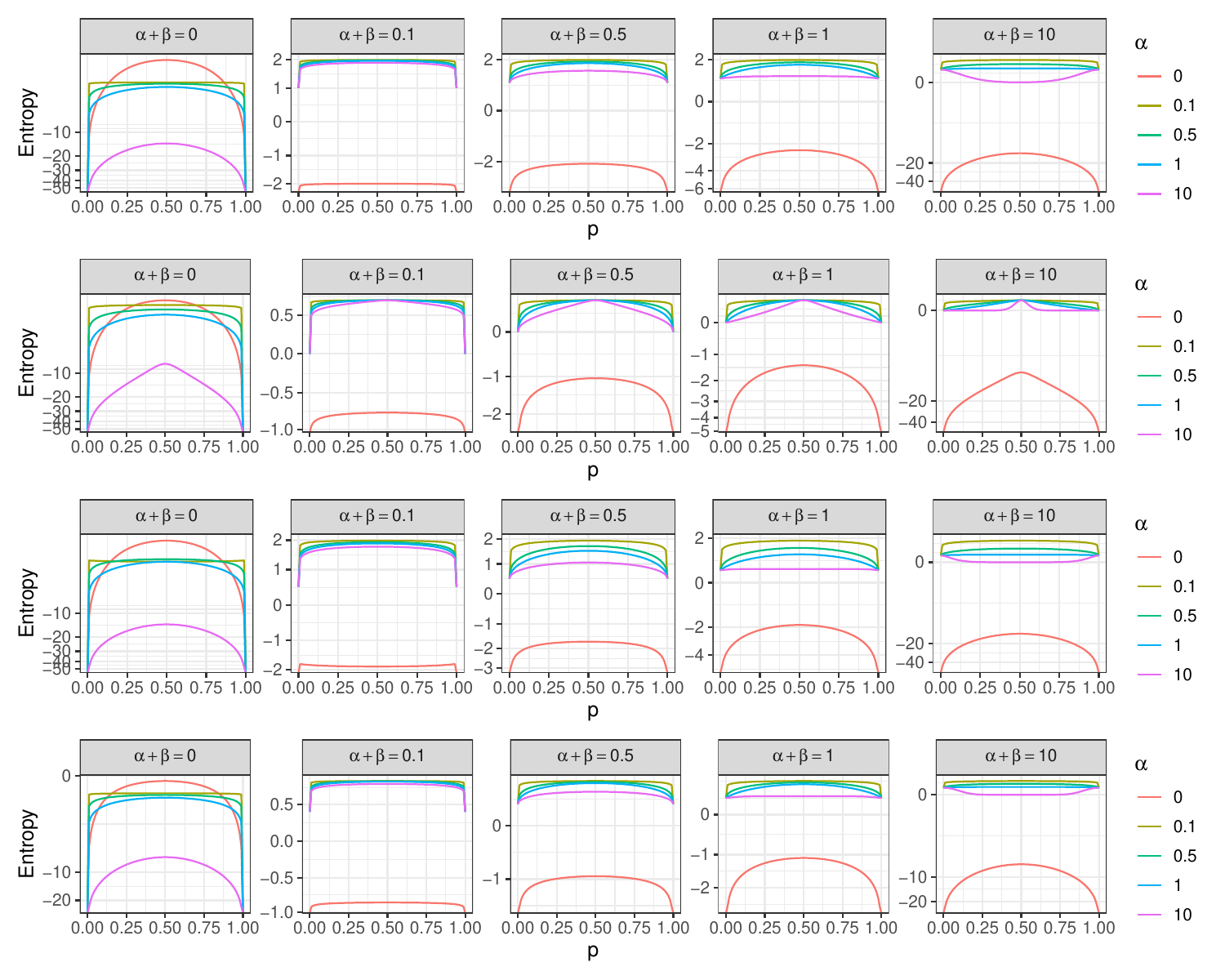}
    \caption{Scaled Generalized Alpha-Beta Entropy for Bernoulli distribution with different choices of generating functions; $\psi(x) = x$, $\psi(x) = \ln(x)$, $\psi(x) = x^2/2$ and $\psi(x) = \ln(x)\Phi(\ln(x)) + \phi(x)$ from top to bottom.}
    \label{fig:gab-entropy-bernoulli}
\end{figure*}

To further understand this measure as a generalized entropy, let us consider a Bernoulli distribution with alphabet $\Acal = \{0, 1\}$ and probability $p(1) = p$, so that $p(0) = (1-p)$. In this case, the GABE for the case $\alpha,\beta, (\alpha+\beta) \neq 0$ reduces to
\begin{equation*}
    \eGAB^{(\alpha,\beta), \psi}(\text{Ber}(p))
    = \frac{\psi(p^\alpha + (1-p)^\alpha)}{\alpha\beta} - \frac{\psi(p^{\alpha+\beta} + (1-p)^{\alpha+\beta})}{(\alpha+\beta)\beta}.
\end{equation*}
\noindent When either $p = 0$ or $p = 1$, the GABE becomes $\psi(1)/\alpha(\alpha+\beta)$, which is the minimum possible value of the entropy in this particular case. Note that, this minimum value depends on the choice of the hyperparameters $\alpha$ and $\beta$. Hence, the entropy measures arising from different choices of these hyperparameters are not directly comparable. Therefore, one may consider scaling the proposed measure given in Eq.~\eqref{eqn:gab-entropy} by multiplying it with $\alpha(\alpha+\beta)$ and treat the resulting quantity as the measure of entropy for the specific case of Bernoulli distribution. In Figure~\ref{fig:gab-entropy-bernoulli}, we show these scaled measures of entropy as functions of the probability $p$, for different combinations of hyperparameters $\alpha$ and $\beta$. We also show different choices of the generating functions that result in different shapes of the resulting entropy measure. As the plots show, except for a few choices of the hyperparameters, the entropy is maximized at $p = 1/2$ in almost all cases. These choices of the hyperparameters are connected to the concavity of the entropy measure, and will be developed later in detail in Theorem~\ref{thm:gab-entropy-concave}.

\section{Theoretical analysis of GAB divergence}\label{sec:gab-analysis}

\subsection{Characterization of the GAB divergence}\label{sec:gab-char}

It is not obvious from the forms given in Eq.~\eqref{eqn:defn-gab-div} to understand the nature of the generating function $\psi$, which ensures GAB divergence to be a well-defined statistical divergence measure. A statistical divergence must satisfy the nonnegativity property for all choices of its arguments. Therefore, we wish to find out the necessary and sufficient conditions on the generating function $\psi$ for which,
\begin{equation*}
    \dGAB^{(\alpha, \beta),\psi}(P, Q) \geq 0,
\end{equation*}
\noindent for all sub-probability distributions $P, Q \in \Mfrak_{\leq 1}(\Acal)$, and equality holds if and only if $P = Q \in \Mfrak_{1}(\Acal)$. It is easy to verify that for any $P \in \Mfrak_{\leq 1}(\Acal)$, $\dGAB^{(\alpha,\beta),\psi}(P, P) = 0$, which follows directly from Eq.~\eqref{eqn:defn-gab-div}. Additionally, if both $\alpha, \beta > 0$, then by a direct consequence of H\"{o}lder's inequality and Jensen's inequality, one can verify that if $\psi$ is monotonically increasing and convex, the form given in Eq.~\eqref{eqn:defn-gab-div} is nonnegative for all pairs of $P$ and $Q$. However, this condition does not characterize the entire class of generating functions $\psi$ for which the resulting GAB divergence is a valid divergence. For example, the LSD is known to be a valid divergence measure~\citep{ghosh2017generalized}, despite having a non-convex generating function $\ln(x)$.

To completely characterize the class of valid generating functions, we follow the innovative approach of~\cite{ray2023functional}, generalizing their results from the context of functional density power divergence (FDPD) class to a significantly larger class of divergence measures defined for all $\alpha, \beta \in \R$. Also, because the limiting forms of GAB given in Eq.~\eqref{eqn:defn-gab-div-edge} are only valid when the generating function $\psi$ is differentiable with a continuous first-order derivative, we shall restrict our attention only to the class of functions $C^1([0,\infty))$ throughout this section. Although it is possible to have a statistically valid divergence with not-so-smooth $\psi$ functions (as in the case of FDPD), the interesting properties of the resulting minimum divergence estimators typically emerge only for sufficiently smooth choices of $\psi$ functions (e.g. $\psi(x) = x, \psi(x) = \ln(x), \psi(x) = x^\phi/\phi$, etc.).


We begin our analysis with the special case $\alpha +\beta = 1$ with $\alpha, \beta \in \R$. Under this choice of the hyperparameters, the GAB divergence reduces to
\begin{equation*}
    \dGAB^{(\alpha,\beta),\psi}(P, Q)
    = \begin{cases}
        [\alpha(1-\alpha)]^{-1} \left( \psi(1) - \psi(\inner{p,q}_{\alpha, 1-\alpha}) \right), & \text{if } \alpha \notin \{ 0, 1 \}, \\
        \psi'(1) \dKL(P, Q),                                                                   & \text{if } \alpha = 1, \beta = 0,    \\
        \psi'(1) \dKL(Q, P),                                                                   & \text{if } \alpha = 0, \beta = 1.
    \end{cases}
\end{equation*}
\noindent The limiting cases clearly lead to valid divergences if and only if $\psi'(1) > 0$, i.e., $\psi$ is strictly increasing at $1$. Interestingly, the same conclusion also holds when both $\alpha = \beta = 0$. We shall demonstrate using the following result that this increasing property of $\psi$ alone is a necessary as well as sufficient for well-definedness of the GAB divergence for the present cases, with $\alpha = \beta = 0$ or $\alpha + \beta = 1$.


\begin{theorem}\label{thm:gab-div-alpha-plus-beta-1}
    Suppose that either $\alpha + \beta = 1$ or $\alpha = \beta = 0$. Then, the GAB divergence is nonnegative if and only if the generating function $\psi$ satisfies the following conditions:
    \begin{enumerate}
        \item When $\alpha \in \{0, 1\}$, $\psi$ is increasing at $1$.
        \item When $\alpha \in (0, 1)$, there exists an $\epsilon > 0$ such that $\psi$ is increasing on $(1-\epsilon, 1]$.
        \item When $\alpha \notin [0, 1]$, there exists an $\epsilon > 0$ such that $\psi$ is increasing on $[1, 1+\epsilon)$.
    \end{enumerate}
\end{theorem}

We now turn our attention to the cases with $\alpha + \beta \notin \{0, 1\}$. Note that, this includes the scenario when exactly one of $\alpha$ or $\beta$ is equal to $0$. A careful inspection of the identities~\eqref{eqn:gab-to-kl-bzero}-\eqref{eqn:gab-to-kl-azero} reveals that, a sufficient condition to ensure nonnegativity of the GAB divergence is to ensure the nonnegativity of each of the terms, i.e., to establish
\begin{align*}
    \psi'(x) x \geq 0, \text{ and, }
    \psi'(x) x (\ln(x) - \ln(y)) - \psi(x) + \psi(y) \geq 0,
\end{align*}
\noindent where $x$ and $y$ are any real numbers. This reduction also uses the fact that the KL divergence is nonnegative. The first of the above conditions translates to the strictly increasing property of the generating function $\psi$, while the second condition is connected to the geometric convexity property of $\psi$. It turns out that these two conditions are both necessary and sufficient for the given forms of the GAB divergence as in Eq.~\eqref{eqn:defn-gab-div}-\eqref{eqn:defn-gab-div-edge} to be a valid statistical divergence measure, including the case when $\alpha$ and $\beta$ are nonzero real numbers satisfying $\alpha + \beta \notin \{0, 1\}$. The formal statement is given below.
\begin{theorem}\label{thm:gab-div-nec-suff}
    Let $\psi \in C^1([0, \infty))$ and $\alpha, \beta$ be real numbers satisfying $(\alpha +\beta) \notin \{0, 1\}$. Define $\Psi(x) := \psi(e^x)$ for all $x \in \R$. Then, the GAB divergence generated by the function $\psi$ is always nonnegative if and only if $\Psi$ is strictly increasing and convex. Also, the GAB divergence is equal to zero if and only if both its arguments match over the sets of sub-probability distributions.
\end{theorem}

When $\alpha = 1$ and $\beta > 0$, then the sufficient and necessary conditions obtained via Theorem~\ref{thm:gab-div-nec-suff} reduces to Proposition 4.1 and Proposition 4.2 of~\cite{ray2023functional}. Although the above result imposes a smoothness assumption of the generating function $\psi$, we require it only for proving the necessary part; see Lemma~\ref{lem:gab-sufficient} in the appendix for further details. It is indeed possible to avoid this smoothness assumption for $\psi$ to establish the necessary condition, but the proof becomes more involved and requires careful modifications of the arguments presented by~\cite{ray2023functional} without any clear statistical benefit. For the sake of simplicity, we present the result under the additional assumption of differentiability.

For the remaining cases where $\alpha = -\beta \neq 0$, the required conditions on $\Psi$ are a bit weaker compared to the ones in Theorem~\ref{thm:gab-div-nec-suff}. In particular, if $\Psi'(0) = 0$, then only a partial monotonicity property of $\Psi$ suffices to ensure the nonnegativity of the GAB divergence.

\begin{theorem}\label{thm:gab-div-nec-suff-case4}
    Let $\psi \in C^1([0, \infty))$ and $\alpha, \beta$ be real numbers satisfying $\alpha = -\beta \neq 0$. Define $\Psi(x) := \psi(e^x)$ for all $x \in \R$. Then, the GAB divergence generated by the function $\psi$ is always nonnegative if and only if $\Psi$ satisfies the following condition: for any $x > y \geq 0$, $\Psi(x) > \Psi(y)$ and $\Psi'(y) > \Psi'(0)$ whenever $\Psi'(0) \neq 0$. However, if $\Psi'(0) = 0$, then the necessary and sufficient condition reduces to only ensuring $\Psi(x) > \Psi(0)$ for any $x \neq 0$.
\end{theorem}

Theorems~\ref{thm:gab-div-alpha-plus-beta-1}, \ref{thm:gab-div-nec-suff} and \ref{thm:gab-div-nec-suff-case4} together completely characterize the necessary and sufficient conditions for Generalized Alpha-Beta (GAB) divergences to be valid statistical divergence measures. Typically, the required condition is the increasing and convexity property of the $\Psi$ function. However, when $\alpha + \beta = 1$ or both $\alpha = \beta = 0$, the required condition is much weaker, as only a locally increasing behavior of $\Psi$ at $x = 0$ suffices. When either $\alpha = 0$ or $\beta = 0$ but not both, the GAB divergence can be seen as a blended version of the KL divergence and a Bregman-type divergence. This perspective intuitively showcases a key property of the GAB divergence family: statistical efficiency is controlled by the KL-divergence part, while the Bregman divergence term serves as a lever to control robustness against outliers. Theorem~\ref{thm:gab-div-nec-suff-case4} also brings out another interesting observation, even when $\Psi$ is not increasing on the entire real line but satisfies a partial increasing and convexity requirements (e.g., $\Psi(x) = x^2$), the resulting GAB divergence is statistically valid when $\alpha + \beta = 0$.


\subsection{Construction of new divergences}\label{sec:gab-construct}

Significant attention has been spent in the past few decades to construct novel families of divergences bearing desirable statistical and information-theoretic properties~\citep{kuus2008extensions, osterreicher2003new}. The results presented in Section~\ref{sec:gab-char} allow us to take on this endeavor within the proposed generalized alpha-beta family of divergence, by using any increasing and convex function as $\Psi$. In the following discussion, we present a few strategies for this construction beyond the already known choices presented in Table~\ref{tab:gab-divergence-families} before.

A general way to build a valid characterizing function is to begin with an $L_{\alpha+\beta}$-integrable density function $f(x)$ (to be precise, any nonnegative $L_{\alpha+\beta}$-integrable function will work), and take $\Psi(x) = \int_{-\infty}^x F(t)dt$ where $F(x)$ is the corresponding cumulative distribution function. Note that, $\Psi(\cdot)$ is increasing and convex by construction. Then, the characterizing function $\psi(x) = \Psi(\ln(x)) = \int_{-\infty}^{\ln(x)} F(t)dt$. For example, if $F(x)$ is the Dirac delta distribution at $x = 0$, the resulting $\psi$ becomes $\ln(x)$, leading to the AC divergence as in~\eqref{eqn:defn-ac-div}. Further, starting with the exponential distribution, we obtain
\begin{equation*}
    \psi_{\text{Exp}}(x)
    = \begin{cases}
        0                & \text{ if } x < 1    \\
        1/x + \ln(x) - 1 & \text{ if } x \geq 1
    \end{cases},
\end{equation*}
\noindent and with the normal distribution, one obtains
\begin{equation}
    \psi_{\text{Normal}}(x)
    = \int_{-\infty}^{\ln(x)} \frac{\ln(x) - u}{\sqrt{2\pi}} e^{-u^2/2}du
    = \ln(x) \Phi(\ln(x)) + \phi(\ln(x)),
    \label{eqn:psi-normal-func}
\end{equation}
\noindent where $\phi(t)$ and $\Phi(t)$ are the standard normal density and cumulative distribution functions respectively. We illustrate a few interesting choices of the generating function $\psi$ in Figure~\ref{fig:gab-div-psi-functions}. Related construction of statistical divergences based on probability distributions is also explored by~\cite{fabian2003core} in the case of Johnson families of distributions.

\begin{figure}[ht]
    \centering
    \includegraphics[width=\linewidth]{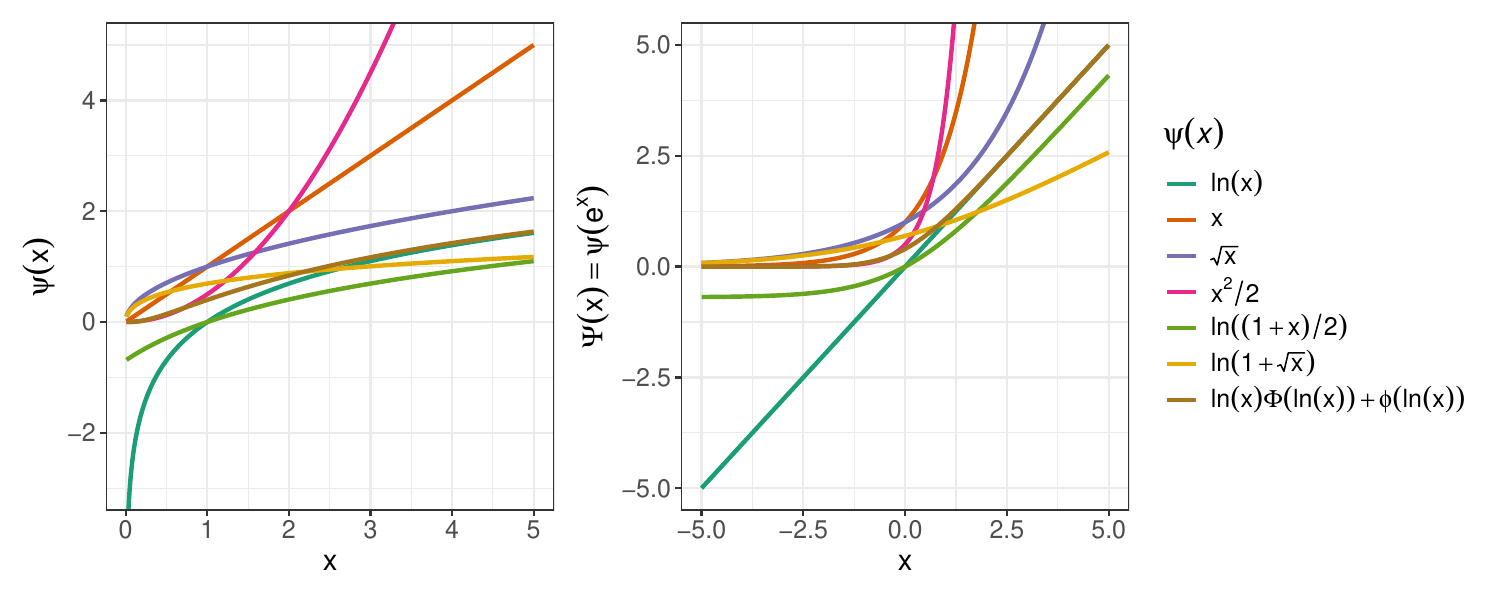}
    \caption{Different choices of generating functions $\psi$ and corresponding $\Psi(x) = \psi(e^x)$ for the GAB divergence.}
    \label{fig:gab-div-psi-functions}
\end{figure}

One may also consider any piece-wise linear function such that it is continuous and the slopes of each of the segments is increasing. Namely, take $\psi(x) = \Psi(\ln x)$ where
\begin{equation*}
    \Psi(x) = \begin{cases}
        a_1 x + b_1        & \text{ if } x \in (-\infty, c_1], \\
        a_2 x + b_2        & \text{ if } x \in (c_1, c_2],     \\
        \dots,                                                 \\
        a_k x + b_k        & \text{ if } x \in (c_{k-1}, c_k], \\
        a_{k+1}x + b_{k+1} & \text{ if } x \in (c_k, \infty)
    \end{cases}
\end{equation*}
\noindent such that $a_1 < a_2 < \dots < a_{k+1}$ and $a_ic_i + b_i = a_{i+1}c_i + b_{i+1}$ for all $i = 1, \dots, k$ to ensure continuity. Note that, although this $\Psi$ function lies in $C^0([0,\infty))$ instead of $C^1([0, \infty))$, its increasing and convex nature ensures that the resulting GAB divergence is nonnegative, expect for the edge cases with $\alpha\beta(\alpha+\beta) = 0$; see Appendix~\ref{appendix:proofs} for a proof. One may regard this choice as an approximation of a smooth generating function. This strategy can also be extended to produce a piece-wise polynomial $\psi$ function satisfying the increasingness, convexity and smoothness requirements, which are particularly useful for adaptive robust statistical inference; see Section~\ref{sec:conclusion} for a brief discussion.

Additionally, if we have two generating functions $\psi_1$ and $\psi_2$ corresponding to two different GAB divergence measures, we can also combine them to construct new classes of GAB divergences, resulting in a blended class of divergences; see~\cite{kuus2003blended} and~\cite{kuchibhotla2019statistical} for related divergences existing in the literature.
\begin{enumerate}
    \item By the necessary conditions given in Section~\ref{sec:gab-char}, we know that the corresponding $\Psi_1$ and $\Psi_2$ will be increasing and convex functions. Clearly, any linear combination of them with positive coefficients will also be increasing and convex. This means, for any $c_1, c_2 > 0$, $\psi := c_1 \psi_1 + c_2 \psi_2$ leads to a valid GAB divergence. One can also establish this by direct verification based on the form~\eqref{eqn:defn-gab-div}.
    \item Another way to combine $\Psi_1$ and $\Psi_2$ is to consider their composition. Since the composition $\Psi_1 \circ \Psi_2$ is also increasing and convex, it follows that $\psi(x) := \Psi_1(\psi_2(x))$ is a valid characterizing function.
\end{enumerate}

\subsection{Properties of the GAB divergence}\label{sec:property}

In this section, we highlight a few properties of the GAB divergence that extends various important existing results. Some properties are given below, which may be verified by straightforward algebraic manipulations from the definitions given in~\eqref{eqn:defn-gab-div} and~\eqref{eqn:defn-gab-div-edge}.
\begin{enumerate}
    \item \textbf{Symmetry (Duality):} As shown before in~\eqref{eqn:gab-symmetry}, the GAB divergence satisfies a symmetry or duality condition, i.e.,
          \begin{equation}
              \dGAB^{(\alpha,\beta),\psi}(P, Q) = \dGAB^{(\beta, \alpha),\psi}(Q, P).
              \label{eqn:gab-symmetry}
          \end{equation}
          Although the GAB divergence is not necessarily symmetric in the sense that in general, $\dGAB^{(\alpha,\beta),\psi}(P, Q)$ and $\dGAB^{(\alpha,\beta),\psi}(Q, P)$ are different, but the symmetry appears whenever $\alpha = \beta$. This symmetrization property along with triangle inequality can convert this divergence into a new metric; such efforts of metrization of divergences have lead to interesting discoveries~\citep{vajda2009metric, okamura2025metrization}.

    \item \textbf{Scaling Property:} Let $c > 0$ such that $cP, cQ \in \Mfrak_{\leq 1}(\Acal)$, and define $\psi_c(x) := \psi(cx)$. Then,
          \begin{equation*}
              \dGAB^{(\alpha,\beta),\psi}(cP, cQ) = \dGAB^{(\alpha,\beta),\psi_{c^{\alpha+\beta}}}(P, Q).
          \end{equation*}
    \item \textbf{Zooming Property:} For any $w \in \R \setminus \{0\}$, the GAB divergence between unnormalized $w$-escort distribution of $P$ and $Q$ satisfies
          \begin{equation}
              \dGAB^{(\alpha,\beta),\psi}(P^w, Q^w) = w^2 \dGAB^{(w\alpha,w\beta),\psi}(P, Q).
              \label{eqn:gab-zoom}
          \end{equation}
\end{enumerate}

The scaling and zooming properties are central to the proofs of the characterization theorems of the GAB divergence. Assume that for a sub-probability distribution $P \in \Mfrak_{\leq 1}(\Acal)$, its unnormalized $\alpha$-escort distribution $P^{\alpha}$ and $\beta$-escort distribution $P^{\beta}$ are well-defined. Then, as a consequence of the zooming property of GAB divergence as in~\eqref{eqn:gab-zoom}, without any loss of generality, any one of the hyperparameters among $\alpha$ or $\beta$ can be made equal to $1$, by modifying the arguments of the divergence suitably and multiplying the final divergence by an appropriate constant. This holds due to the following relationship,
\begin{equation}
    \dGAB^{(\alpha,\beta),\psi}(P, Q)
    = \dfrac{1}{\beta^2} \dGAB^{(\alpha/\beta,1),\psi}(P^\beta, Q^\beta)
    = \dfrac{1}{\alpha^2}\dGAB^{(1,\beta/\alpha),\psi}(P^\alpha, Q^\alpha).
    \label{eqn:gab-reduction-to-1}
\end{equation}

The following results establish another important property for the GAB divergence generalizing Propositions 3 and 4 of~\cite{kumar2015minimization}.

\begin{proposition}[Lower Semicontinuity in first argument]\label{thm:gab-semicont}
    Suppose $\psi \in C^0([0, \infty))$ and $\alpha + \beta > 0$ and $\alpha\beta \neq 0$. Then the GAB divergence is lower-semicontinuous in its first argument, i.e., if $P_n$ is a sequence of probability distributions converging to $P$ in $L_{\alpha+\beta}(\mu)$, then
    \begin{equation}
        \liminf_{n\rightarrow \infty} \dGAB^{(\alpha,\beta),\psi}(P_n, Q) \geq \dGAB^{(\alpha,\beta),\psi}(P, Q),
        \label{eqn:liminf-gab-div}
    \end{equation}
    \noindent for any $Q \in L_{\alpha+\beta}(\mu)$. If $\beta > 0$, then the semi-continuity can be strengthened to continuity.
\end{proposition}
\noindent While the above establishes the continuity of the GAB divergence in its first argument, an analogous result holds for the function $q \mapsto \dGAB^{(\alpha,\beta),\psi}(p, q)$ for a fixed $p$, establishing semi-continuity in the second argument. It follows trivially from Proposition~\ref{thm:gab-semicont} and the symmetry property~\eqref{eqn:gab-symmetry}.

\begin{proposition}[Lower Semicontinuity in second argument]
    Suppose $\psi \in C^0([0, \infty))$ and $\alpha + \beta > 0$ and $\alpha\beta \neq 0$. Then the GAB divergence is lower-semicontinuous in its second argument, i.e., if $Q_n$ is a sequence of probability distributions converging to $Q$ in $L_{\alpha+\beta}(\mu)$, then
    \begin{equation}
        \liminf_{n\rightarrow \infty} \dGAB^{(\alpha,\beta),\psi}(P, Q_n) \geq \dGAB^{(\alpha,\beta),\psi}(P, Q),
        \label{eqn:liminf-gab-div-second}
    \end{equation}
    \noindent for any fixed $P \in L_{\alpha+\beta}(\mu)$. If $\alpha > 0$, then the semi-continuity can be strengthened to continuity.
\end{proposition}

A fundamental property that allows a statistical divergence to be suitable for robust inference is the Pythagorean identity (or inequality). It turns out that, the proposed GAB divergence satisfies an approximate version of this property. Let $p_\epsilon$ be a $(\alpha,\epsilon)$-convex combination of two probability densities $p_0$ and $\delta$ over the same alphabet set $\Acal$, i.e.,
\begin{equation}
    p_\epsilon^\alpha = (1-\epsilon)p_0^\alpha + \epsilon \delta^\alpha, \
    \alpha > 0, \ \epsilon \in [0, 1].
    \label{eqn:alpha-convex-comb}
\end{equation}
\noindent The choice of such a nonlinear combination over the standard Huber contamination model has been previously adopted for the analysis of projection theorems and escort distributions, e.g., see~\cite{amari2010information, ghosh2021scale, gayen2021projection}. Let us also define, for any two sub-probability densities $p$ and $q$,
\begin{equation}
    \tilde{d}_{\psi}^{(\alpha,\beta)}(p, q) := -\frac{1}{\alpha\beta} \left[ \psi\left( \inner{p,q}_{\alpha,\beta} \right) -\lambda_\beta \psi\left( \norm{q}_{\alpha+\beta}^{\alpha+\beta} \right)\right],
    \label{eqn:dtilde-defn}
\end{equation}
\noindent when $\alpha\beta(\alpha+\beta)\neq 0$ and $\lambda_\beta = \alpha/(\alpha+\beta)$. It is easy to see that,
\begin{equation}
    \dGAB^{(\alpha,\beta),\psi}(p,q) = -\tilde{d}_{\psi}^{(\alpha,\beta)}(p,p) + \tilde{d}_{\psi}^{(\alpha,\beta)}(p, q)
    \label{eqn:dtilde-to-gabdiv}
\end{equation}
\noindent Now, we present a lemma that connects $\tilde{d}_{\psi}^{(\alpha,\beta)}(p_\epsilon,q)$ to $\tilde{d}_{\psi}^{(\alpha,\beta)}((1-\epsilon)^{1/\alpha} p,q)$, generalizing Lemma 3.1 of~\cite{fujisawa2008robust}, which provides the pathway to formally develop the approximate Pythagorean identity.

\begin{lemma}\label{lem:dtilde-error}
    Let $\alpha\beta(\alpha+\beta)\neq 0$, $p_0, q, \delta \in L_{\alpha+\beta}(\mu)$ and $\psi \in C^1((0,\infty))$ be a valid characterizing function for the GAB divergence. Define $p_\epsilon$ as in~\eqref{eqn:alpha-convex-comb}. Then,
    \begin{equation}
        \tilde{d}_{\psi}^{(\alpha,\beta)}(p_\epsilon, q)
        = \tilde{d}_{\psi}^{(\alpha,\beta)}((1-\epsilon)^{1/\alpha} p_0, q) + O(\epsilon \inner{\delta, q}_{\alpha,\beta}).
        \label{eqn:dtilde-error}
    \end{equation}
    \noindent and,
    \begin{equation}
        \tilde{d}_{\psi}^{(\alpha,\beta)}(p_0, q) = \tilde{d}_{\psi}^{(\alpha,\beta)}((1-\epsilon)^{1/\alpha} p_0, q) + O(\ln(1-\epsilon)).
        \label{eqn:dtilde-error-2}
    \end{equation}
\end{lemma}

\begin{theorem}[Approximate Pythagorean Identity]\label{thm:pythagorean}
    Let $\alpha\beta(\alpha+\beta) \neq 0$, $p_0,q,\delta \in L_{\alpha+\beta}(\mu)$, and $\psi \in C^1((0, \infty))$ be a valid characterizing function for the GAB divergence. Define $p_\epsilon$ as in~\eqref{eqn:alpha-convex-comb}. Then, an approximate Pythagorean relation holds between $p_0, p_\epsilon$ and $q$ in the metric defined through the GAB divergence, as given by
    \begin{equation*}
        \Delta(p_\epsilon, p_0, q)
        := \dGAB^{(\alpha,\beta),\psi}(p_\epsilon, q) - \dGAB^{(\alpha,\beta),\psi}(p_\epsilon, p_0) - \dGAB^{(\alpha,\beta),\psi}(p_0, q) = O(\epsilon v_\delta) + O(\ln(1-\epsilon)),
        \label{eqn:pythagorean-error}
    \end{equation*}
    \noindent where $v_\delta = \max\{ \inner{\delta, p_0}_{\alpha,\beta}, \inner{\delta, q}_{\alpha,\beta} \}$.
\end{theorem}

Note that, Theorem~\ref{thm:pythagorean} is a generalization of the Pythagorean identity proved by~\cite[Theorem 3.2]{fujisawa2008robust} and of~\cite[Theorem 3.3]{kuchibhotla2019statistical} for the present case of GAB divergence. Compared to their results, our result includes an additional bias term of order $\ln(1-\epsilon)$ to the discrepancy $\Delta(p_\epsilon, p_0, q)$. A careful consideration of the proofs of Lemma~\ref{lem:dtilde-error} and Theorem~\ref{thm:pythagorean} reveal that this logarithmic term can be quantified as $\left[ \Psi'\left( y_0 \right) - \Psi'\left( \ln(1-\epsilon_0) + x_0 \right) \right] \ln(1-\epsilon)$, for any $\epsilon \leq \epsilon_0$, a pre-specified level of contamination, where $x_0 = \ln(\inner{p_0, q}_{\alpha,\beta})$ and $y_0 = \ln(\norm{p_0}_{\alpha+\beta}^{\alpha+\beta})$. In the case of logarithmic alpha-beta divergence, we have $\Psi(x) = x$, and hence this additional bias term equals $0$, leading to the exact Pythagorean identity.

\section{Theoretical analysis of GAB entropy}\label{sec:theory-gab-entropy}

\subsection{Properties of GAB entropy}\label{sec:gabe-properties}

We now redirect our efforts towards understanding various properties of the GAB entropy (GABE). As indicated in Section~\ref{sec:gab-entropy}, much of this discussion will be restricted to the case when $\Acal$ is finite. It can be easily seen from the definition in Eq.~\eqref{eqn:gab-entropy} that the GAB entropy satisfies the following properties.
\begin{enumerate}
    \item \textbf{Continuity:} For a finite alphabet $\mathcal{A}$, the GAB entropy is continuous in the $p_i = p(a_i)$ for all $i = 1, 2, \dots, \vert \Acal\vert$.
    \item \textbf{Symmetry:} $\eGAB^{(\alpha,\beta),\psi}(P)$ is a symmetric function of $\{ p_1, p_2, \dots p_{\vert \Acal\vert} \}$ for any $P \in \Mfrak_{\leq 1, \vert \Acal\vert}$.
    \item \textbf{Decisivity:} For any degenerate distribution $\delta_x$ at $x \in \mathcal{A}$, $\eGAB^{(\alpha,\beta),\psi}(\delta_x) = 0$.
    \item \textbf{Expandability:} For any sub-probability distribution $P := (p_1, \dots, p_n) \in \Mfrak_{\leq 1, n}$, we have
          \begin{equation*}
              \eGAB^{(\alpha,\beta),\psi}(P) = \eGAB^{(\alpha,\beta),\psi}(P^\ast),
          \end{equation*}
          \noindent where $P^\ast := (p_1, \dots, p_n, 0) \in \Mfrak_{\leq 1, n+1}$ is another sub-probability distribution over a new alphabet set having $(n+1)$ elements.
    \item \textbf{Maximum value:} For a finite alphabet $\Acal$ with $\vert \Acal\vert = n$, the maximum value of GABE measure is $\psi(n^{\alpha+\beta-1})/\alpha(\alpha+\beta)$. This is a simple consequence of the nonnegativity of the GAB divergence and the construction of the GABE as given in Eq.~\eqref{eqn:gab-entropy}.
\end{enumerate}




\subsection{Maximum entropy principle}\label{sec:maxent-distn}

The maximum entropy principle is a fundamental concept having various applications in statistics, information theory and statistical physics. A key property of any entropy measure is its concavity which allows one to apply the maximum entropy principle. From a statistical point of view, the maximum entropy principle can be viewed as a special case of the minimum divergence estimation principle. Fixing the reference measure $Q$ to be the uniform measure (for a finite alphabet set $\Acal$), minimizing the divergence to a uniform measure is mathematically equivalent to maximizing the associated entropy. The concavity of the entropy (and convexity of the divergence) ensures that these optimization problems are well-defined and have unique optimizer in standard settings. In the following proposition, we derive the exact set of conditions under which the concavity of the GAB entropy follows. A similar set of conditions for the logarithmic norm entropy is present in~\cite{ghosh2021scale}. The following proposition outlines the choices of $\alpha, \beta$ and $\psi$ for which this concavity property can be attained.

\begin{proposition}\label{thm:gab-entropy-concave}
    The GAB entropy $\eGAB^{(\alpha,\beta),\psi}$ is $\gamma$-quasi-concave (see Definition in Appendix~\ref{defn:quasi-convex}) if any of the following conditions hold
    \begin{enumerate}
        \item $\psi(x) = x$, with either $\alpha > 0, \beta \in (-\alpha, 0), \gamma \in [\alpha+\beta, \alpha]$ or $\alpha < 0, \beta \in (0, -\alpha), \gamma \in [\alpha, \alpha+\beta]$.
        \item $\psi$ is convex with $\alpha\beta < 0, \beta(\alpha+\beta) > 0$ and either of (i) $\gamma \in (-\infty, \min\{\alpha,\alpha+\beta\})$, (ii) $\gamma \in [0, \max\{\alpha, \alpha+\beta\}]$, (iii) $\gamma \in (-\infty, \alpha)$ if $\alpha < 0$, (iv) $\gamma \in (-\infty, \alpha +\beta)$ if $\alpha+\beta < 0$.
        \item $\psi$ is concave with either of (i) $\alpha, \beta < 0$ and any $\gamma \geq 0$. (ii) $\alpha(\alpha+\beta) < 0$ and $\gamma \geq \max\{\alpha, \alpha+\beta\}$. (iii) $\alpha < 0, \beta \in (0, -\alpha)$ and any $\gamma \geq 0$.
    \end{enumerate}
\end{proposition}

A consequence of Proposition~\ref{thm:gab-entropy-concave} is that under the same conditions, the GAB divergence is $\gamma$-quasi-convex in its first argument. This follows directly from the relationship that $\eGAB^{(\alpha,\beta),\psi}(P) = -\dGAB^{(\alpha,\beta),\psi}(P, Q_0) + h(Q_0)$ where $Q_0$ is the fixed uniform measure on the alphabet $\mathcal{A}$ and $h(Q_0)$ is a quantity free of $P$. On the other hand, because of the symmetry property~\eqref{eqn:gab-symmetry}, exchanging the role of $\alpha$ and $\beta$ in the conditions presented in Proposition~\ref{thm:gab-entropy-concave}, we obtain a set of sufficient conditions under which the GAB divergence is $\gamma$-quasi-convex in its second argument. These conditions are appropriately summarized in Figure~\ref{fig:gab-convexity}.


\begin{figure}[t]
    \centering
    \includegraphics[width=\linewidth]{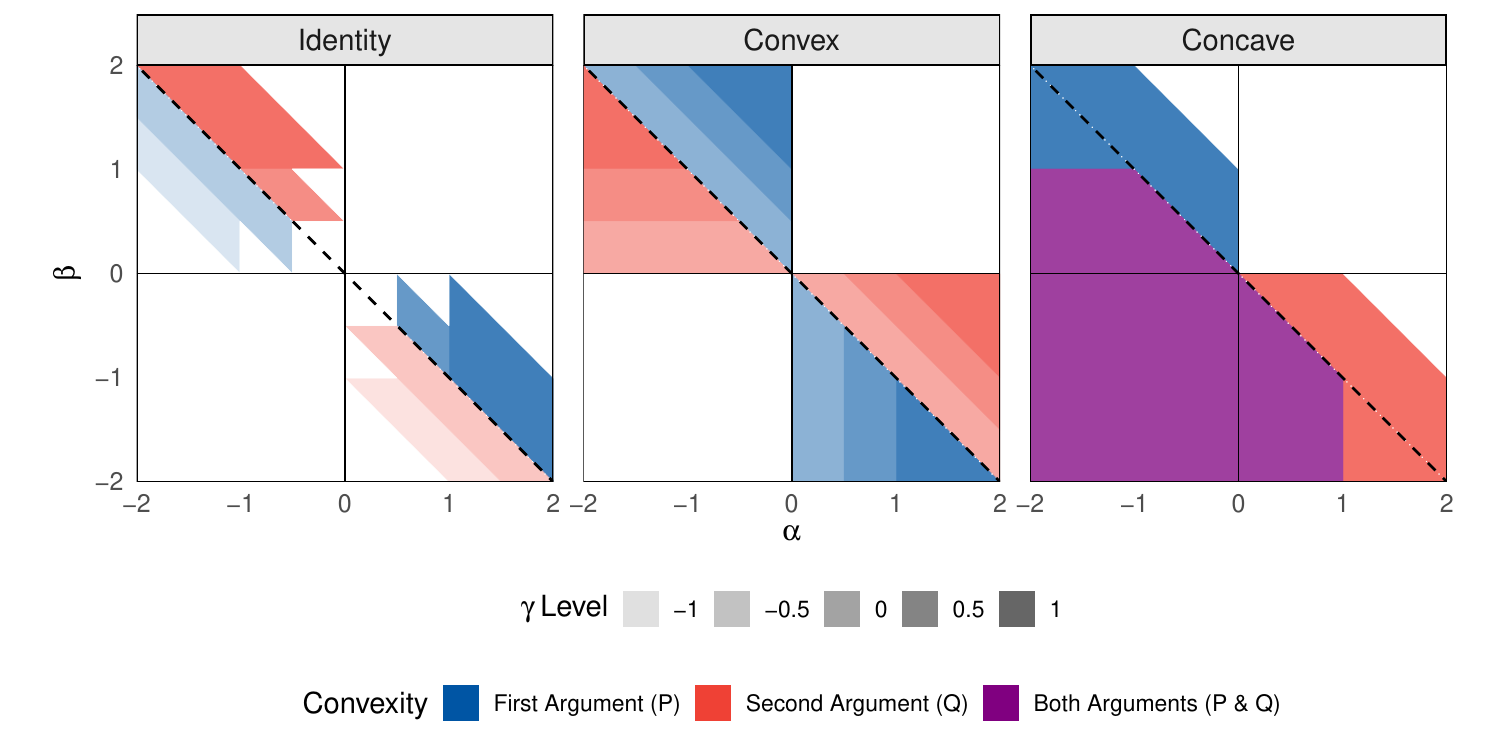}
    \caption{The regions indicating when the GAB divergence is $\gamma$-quas-convex in its first or second argument or both, provided different curvature conditions on $\psi$ function (indicated above each sub-figure) and at varying levels of $\gamma$. The region for $\psi(x) = x$ indicates the valid regions in addition to the regions for convex or concave $\psi$ functions.}
    \label{fig:gab-convexity}
\end{figure}

While the classical literature contains several instances where the maximum entropy distribution has been investigated under linear constraints (which translates to various moment-type conditions in statistics), modern literature in information theory generalizes the constraints to consider normalized $\beta$-expectations~\citep{ghosh2021scale}. Here, we consider the constraints for a probability distribution $P \in \Mfrak_{1, n}$ over a finite alphabet set $\Acal$ given by
\begin{equation}
    \dfrac{\sum_{i=1}^n g_r(a_i)p_i^\alpha }{\sum_{i=1}^n p_i^\alpha} = G_r,
    \label{eqn:maxent-constraints}
\end{equation}
\noindent for $r = 1, 2, \dots, m$, where $g_1, g_2, \dots, g_m: \Acal \rightarrow \R$ are $m$ given functions on $\Acal$ and $G_1, G_2, \dots, G_m$ are given fixed constants. In applications of statistical physics, these constraints in~\eqref{eqn:maxent-constraints} may indicate different physical and chemical properties of particles under a controlled environment in a closed system. The maximum entropy principle, in this case, tells us that the underlying system would aim to maximize its entropy, and the maximizing probability distribution can characterize the behavior (e.g. position, velocity) of the particles in the closed system~\citep{jaynes1957information,kapur1992entropy}.

Before presenting the result pertaining to the maximum entropy distribution, we define two key notations as follows. For a probability distribution $P \in \Mfrak_{1, n}$, we denote
\begin{equation}
    c_1(P) := \frac{\alpha+\beta}{\alpha} \left[ \ln\left( \norm{p}_{\frac{\alpha+\beta}{\alpha}} \right) - \ln\left( \norm{p}_{1/\alpha} \right) \right], \
    c_2(P) := -\ln(\norm{p}_{1/\alpha}), \
    c_3(P) := \sum_{i=1}^n p_i \ln(p_i).
    \label{eqn:c1p-defn}
\end{equation}
\noindent Equipped with this notation, we now present an implicit form of the maximizer of GABE under the constraints given in~\eqref{eqn:maxent-constraints}. In view of~\eqref{eqn:maxent-constraints}, it is only interesting to consider the case $\alpha \neq 0$. Additionally, since $\alpha + \beta = 1$ case always leads to the power divergence irrespective of the choice of $\psi$ (see the proof of Theorem~\ref{thm:gab-div-alpha-plus-beta-1}), we additionally restrict our attention to the case when $\alpha + \beta \neq 1$.

\begin{theorem}\label{thm:maxent-dist}
    Let $\alpha, \beta$ be real numbers such that $\alpha \neq 0, (\alpha+\beta) \neq 1$. For a finite alphabet set $\Acal$ with $\vert \Acal\vert = n$, the probability distribution $P = (p_1, \dots, p_n) \in \Mfrak_{1, n}$ which maximizes the GAB Entropy~\eqref{eqn:gab-entropy} with a valid generating function $\psi \in C^2([0,\infty))$ subject to the constraints~\eqref{eqn:maxent-constraints} is given by $Q^{[1/\alpha]}$, the $1/\alpha$-escorted version of the probability distribution $Q = (q_1, \dots, q_n)$ satisfying the following fixed point proportionality relationship.
    \begin{equation}
        q_i \propto \begin{cases}
            \left[  e^{c_2(Q)/\alpha} (\Psi'(c_1(Q)) - \Psi'(c_2(Q))) + \alpha\beta \sum_{r=1}^m \dfrac{\lambda_r g_r(a_i)}{q_i^{1/\alpha-1}} \right]^{\frac{\alpha}{\alpha+\beta-1}}
                                                                                                                                                                                                                     & \text{ if } \beta(\alpha+\beta) \neq 0, \\
            \exp\left[ \dfrac{\alpha^2}{\Psi'(c_2(Q))} \left\{ \sum_{r=1}^m \lambda_r g_r(a_i) + \frac{\Psi''(c_2(Q))}{e^{-c_2(Q)/\alpha}} \left( c_2(Q) + c_3(Q) \right) q_i^{\frac{1}{\alpha}-1} \right\} \right], & \text{ if } \beta = 0, \alpha \neq 0,   \\
            \left[ \sum_{r=1}^m \lambda_r g_r(a_i) + q_i^{1/\alpha-1} e^{c_2(Q)/\alpha} (n\Psi'(0) - \alpha\Psi'(c_2(Q)))  \right]^{-1},
                                                                                                                                                                                                                     & \text{ if } \beta = -\alpha \neq 0.
        \end{cases}
        \label{eqn:maxent-dist-type}
    \end{equation}


    \noindent In all cases, $\Psi(x) := \psi(e^x)$, $c_1(Q), c_2(Q)$ and $c_3(Q)$ are as defined in Eq.~\eqref{eqn:c1p-defn}, and the unknown constants $\lambda_1, \dots, \lambda_r$ are uniquely determined by the constraints given in~\eqref{eqn:maxent-constraints}.
\end{theorem}

Theorem~\ref{thm:maxent-dist} may be regarded as a generalization of the results on the maximum entropy distribution obtained by~\cite{ghosh2021scale}. Specifically, with the choice of $\Psi(x) := \psi(e^x) = x$ for the logarithmic Alpha-Beta divergence, we find that $\Psi'(x)=1$ and $\Psi''(x) = 0$ for any $x \in \R$, and hence the self-proportionality relationships given in Eq.~\eqref{eqn:maxent-dist-type} become simplified, namely one can obtain a closed form solutions to all of these fixed-point proportionality equations. For example, when $\beta(\alpha+\beta) \neq 0$, the entropy maximizing distribution is $1/\alpha$-escorted version of a distribution from the power-law family, while this maximizer becomes a member of the classical exponential family at $\beta = 0 \neq \alpha$; interested readers are referred to~\cite{ghosh2021scale} for further details.

In general, it is not known whether such a fixed-point proportionality relationship has a solution or if the solution is unique (up to a constant scaling factor) even if it exists. However, when the GAB entropy is concave, the uniqueness of the maximum entropy distribution is assured provided that the constraints~\eqref{eqn:maxent-constraints} have a feasible solution. Under such a scenario, one can numerically compute the maximum entropy distribution by following an iterative approach: start with any feasible solution $P_0$, compute its $\alpha$-escort distribution $Q_0$, and proceed by refining the estimate of $Q$ by carrying out one of the fixed-point equation~\eqref{eqn:maxent-dist-type} iteratively until convergence. While closed-form solutions are intractable, a careful look at Eq.~\eqref{eqn:maxent-dist-type} reveals many useful structures of the maximum entropy distribution. Let us indicate $K_1, K_2$ as some generic constants and $h(a_i) = \sum_{r=1}^m \lambda_r g_r(a_i)$. Then, structurally, the three different cases in Eq.~\eqref{eqn:maxent-dist-type} have the forms
\begin{align*}
    q_i^{(\alpha+\beta-1)/\alpha} & \propto K_1 + K_2 \frac{h(a_i)}{q_i^{1/\alpha - 1}}, \ \text{ if } \beta(\alpha+\beta) \neq 0,               \\
    q_i                           & \propto \exp\left[ K_1 h(a_i) + K_2 q_i^{1/\alpha - 1} \right], \ \text{ if } \beta = 0, \alpha \neq 0,      \\
    q_i                           & \propto \left[ h(a_i) + K_1 q_i^{1/\alpha - 1} \right]^{-1}, \ \text{ if } \beta \neq 0, \alpha + \beta = 0.
\end{align*}
\noindent In all these cases, if $h(a_i)$ are same for all $a_i \in \Acal$, then the maximum entropy distribution is the uniform distribution over $\Acal$ as expected. Otherwise, the shape of the maximum entropy distribution is largely controlled by the shape of $h(a_i)$, i.e., the moment restrictions. It results in an implicit generalization of the exponential family of distributions, which includes Tsallis-type power law families and deformed exponential distributions as special cases~\citep{naudts2004estimators, tsallis2009introduction, amari2011geometry}.

Although the above discussion pertains to the discrete distributions with finite alphabet set, Theorem~\ref{thm:maxent-dist} can be extended to discrete distributions with countably infinite support, and continuous distributions, provided all the required integrals (or sums) as in~\eqref{eqn:c1p-defn} remain finite and well-defined. For the finite alphabet case, the finiteness of these integrals (or sums) is automatic.

\section{Application in Robust Parametric Inference}\label{sec:robust-parametric-inference}

\subsection{The minimum GAB estimators}\label{sec:min-gab-est}

Consider a scenario where the true data-generating distribution $G$ has density $g$, and we fit a parametric model $f_\theta$ governed by $\theta \in \Theta$. The minimum GAB divergence estimate (MGABDE) of the parameter $\theta$ is defined as
\begin{equation*}
    \hat{\theta}^{(\alpha,\beta), \psi} := \argmin_{\theta \in \Theta} \dGAB^{(\alpha,\beta), \psi}(\hat{g}, f_\theta),
\end{equation*}
\noindent where $\hat{g}$ is an empirical approximation of the true $g$ based on the random sample $X_1, \dots, X_n$. The minimizing density $f_{\hat{\theta}^{(\alpha,\beta),\psi}}$ is also called the reverse projection of $\hat{g}$~\citep{gayen2021projection}. For real numbers $\alpha, \beta$ satisfying $\alpha\beta(\alpha+\beta) \neq 0$, minimizing the GAB divergence given in Eq.~\eqref{eqn:defn-gab-div} yields an estimating equation of the form
\begin{equation}
    \frac{\int f_\theta^{\alpha+\beta} u_\theta}{\xi(\norm{f_\theta}_{\alpha+\beta}^{\alpha+\beta}) } = \frac{\int f_{\theta}^\beta \hat{g}^\alpha u_\theta }{\xi(\inner{f_\theta, \hat{g}}_{\beta,\alpha})},
    \label{eqn:general-est-eqn}
\end{equation}
\noindent where $u_\theta(x) = \nabla_\theta \ln(f_\theta(x))$ is the score function,
and $\xi(x) = 1/\psi'(x)$; see \cite{windham1995robustifying} and \cite{jones2001comparison} for related calculations. Structurally, Eq.~\eqref{eqn:general-est-eqn} represents a generalized likelihood equation that matches a modified score functional $U_\alpha(G) = \int f_\theta^\beta g^\alpha u_\theta / \xi(\inner{f_\theta, g}_{\beta,\alpha})$ between its population and empirical counterparts.

While the hyperparameter $\alpha$ controls the baseline geometry of the cross-entropy between model density and the true density, we can isolate the mechanical intuition of $\psi$ by temporarily restricting our attention to the $\alpha = 1$ case. This case also simplifies the integrals $\inner{f_\theta, g}_{\beta,1}$ to be estimated via an empirical average $n^{-1}\sum_{i=1}^n f_\theta^\beta(X_i)$, allowing us to cleanly write the estimating equations as
\begin{equation}
    \frac{\int f_\theta^{\beta+1} u_\theta}{ \xi(\norm{f_\theta}_{\beta+1}^{\beta+1})} = \frac{n^{-1}\sum_{i=1}^n f_\theta^\beta(X_i) u_\theta(X_i) }{\xi(n^{-1}\sum_{i=1}^n f_\theta^\beta(X_i))}.
    \label{eqn:general-est-eqn-alpha1}
\end{equation}
\noindent This formulation explicitly decouples two distinct mechanisms of robustness. The exponent $\beta$ acts at the observation level. An anomalous data point $X_i$ yields a small model density $f_\theta(X_i)$, and $\beta > 0$ dictates the aggressive downweighting of the score $u_\theta(X_i)$ for that specific outlier. In contrast, the term, $n^{-1}\sum_{i=1}^n f_\theta^\beta(X_i)$, acts as a measure of aggregate goodness-of-fit across the entire sample. A significant drop in this empirical average relative to its theoretical expectation $\norm{f_\theta}_{\beta+1}^{\beta+1}$ indicates gross model misspecification that observation-level downweighting alone cannot rectify. The strict increasingness of $\psi$ ensures that this normalization factor remains positive, which is key to maintaining proper weighting. The curvature of the generating function $\psi$ strictly governs this macroscopic behavior. If $\psi$ is chosen to be strictly convex, its derivative $\psi'$ is a positive and increasing function, which renders $\xi(x)$ a strictly decreasing function. Consequently, when the aggregate model fit is poor (yielding a small empirical overlap), the denominator $\xi(n^{-1}\sum_{i=1}^n f_\theta^\beta(X_i))$ becomes correspondingly large, leading to a very small right-hand side in~\eqref{eqn:general-est-eqn-alpha1}. This stops any $\hat{\theta}$ leading to a poor fit from acting as viable solutions for the estimating equation. This global robustness behavior manifests through their high asymptotic breakdown points, as established in~\cite{roy2026gaboptimal}. In fact, achieving optimal statistical efficiency while securing this structural robustness forces us to look beyond the standard choices of $\psi(x) = x$ (AB-divergence) or $\psi(x) = \ln(x)$ (logarithmic AB-divergence), firmly motivating the study of the broader GAB class.

\subsection{Projection theorems and Conjugate Family}\label{sec:projection-theorems}

We now move on to a projection result for the GAB divergence to showcase that, under a suitable choice of the parametric model family, the MGABDE structurally simplifies to a scaled generalized method of moments estimator. For the ease of exposition, we restrict our attention to the case when the model family is discrete and the alphabet set $\Acal$ is finite, and the special case with $\alpha = 1$ and $\beta \neq 0$ corresponding to the functional geometry of the density power divergence. As noted before, the following discussions can be extended to discrete distributions with infinite support or continuous model families, provided that the necessary sums (or integrals) are well-defined and finite.

Under the discrete setting with finite alphabet set, the fixed-point proportionality relationship in~\eqref{eqn:maxent-dist-type} yields a parametric family of the form
\begin{equation}
    f_{\theta}(a_i) \propto \left[ \Psi'(\ln \norm{f_\theta}_{1+\beta}^{1 + \beta}) - \Psi'(0) + \beta \sum_{r=1}^m \theta_r g_r(a_i) \right]^{1/\beta},
    \label{eqn:maxent-dist-type1-alpha1}
\end{equation}
\noindent where $\theta = (\theta_1, \dots, \theta_m)$ is the vector of parameters. Defining the normalizing constant as $K(\theta) = \Psi'(\ln \norm{f_\theta}_{1+\beta}^{1+\beta}) - \Psi'(0)$, the density satisfies the algebraic identity $f_\theta^\beta(x) = K(\theta) + \beta \sum_{r=1}^m \theta_r g_r(x)$. Differentiating this identity with respect to $\theta$ yields a score function $u_\theta(x) = \nabla \ln f_\theta(x)$ characterized by
\begin{equation*}
    \beta f_\theta^\beta(x) u_\theta(x) = \nabla K(\theta) + \beta g(x),
\end{equation*}
\noindent where $g(x) = (g_1(x), \dots, g_m(x))\tr$ and $\nabla$ denotes the gradient operator with respect to $\theta$. For this choice of parametric model, the resulting estimating equation~\eqref{eqn:general-est-eqn-alpha1} can be simplified as
\begin{align}
             & \frac{1}{\xi(\norm{f_\theta}_{\beta+1}^{\beta+1})} \int f\theta^{\beta+1}(x) u_\theta(x) dx = \frac{1}{\xi(n^{-1}\sum_{i=1}^n f_\theta^\beta(X_i))} \frac{1}{n} \sum_{i=1}^n f_\theta^\beta(X_i) u_\theta(X_i) \nonumber                                                \\
    \implies & \frac{1}{\xi(\E_{f_\theta}[f_\theta^\beta(X)])} \E_{f_\theta}\left[ \beta^{-1}\nabla K(\theta) + g(X) \right] = \frac{1}{\xi(\E_n[f_\theta^\beta(X)])} \E_n\left[ \beta^{-1}\nabla K(\theta) + g(X) \right] \nonumber                                                   \\
    \implies & \frac{\E_n[g(X)]}{\xi(\E_n[f_\theta^\beta(X)])} = \frac{\E_{f_\theta}[g(X)]}{\xi(\E_{f_\theta}[f_\theta^\beta(X)])} + \frac{\nabla K(\theta)}{\beta} \left[ \psi'\left( \E_{f_\theta}[f_\theta^\beta(X)] \right) - \psi'\left( \E_n[f_\theta^\beta(X)] \right) \right],
    \label{eqn:projection-theorem-gmm}
\end{align}
\noindent where $\E_n$ indices the expectation operator with respect to the empirical distribution of the observed sample. Equation~\eqref{eqn:projection-theorem-gmm} reveals a structural alignment with the generalized method of moments, obstructed only by a bias term that arises from the non-orthogonality of projections within the curved geometry induced by $\psi$. This bias term vanishes entirely if and only if one of two global conditions holds for any arbitrary empirical sample: (i) $K(\theta) = \Psi'(\ln \norm{f_\theta}_{1+\beta}^{1+\beta}) - \Psi'(0)$ is constant. (ii) $\psi'( \E_{f_\theta}(\norm{f_\theta}_\beta^\beta)) = \psi'( \E_n(\norm{f_\theta}_\beta^\beta))$. Straightforward calculations to solve these differential equations force either $\Psi(x) = x$ or $\psi(x) = x$, up to some affine transformations. Consequently, when fitting the entropy-maximizing conjugate family, the estimating equation reduces to an exact generalized method of moments only for the specific cases of the standard AB-divergence and the logarithmic AB-divergence. This exact orthogonality property is precisely why the $O(\ln(1-\epsilon))$ bias term vanishes from the approximate Pythagorean identity (Theorem~\ref{thm:pythagorean}) strictly under these two geometries. In essence, this reveals that the exact mathematical conjugacy and computational tractability of standard projection theorems can be intentionally traded for the global robustness properties discussed above by suitably modifying the $\psi$-function.

\subsection{Illustrations under a discrete model: Geometric distribution}\label{sec:empirical-geom}

We consider the problem of estimating the success probability $p$ of a Geometric distribution. Let the parametric model be denoted by a probability mass function $f_p \sim \text{Geo}(p)$ with the true parameter $p_0$. We take the data-generating mechanism following a Huber contamination model as
\begin{equation*}
    g_\epsilon \sim (1-\epsilon)f_{p_0} + \epsilon \delta_{x_0},
\end{equation*}
\noindent where $\delta_{x_0}$ is a degenerate probability mass function at $x = x_0$. Our initial experiment aims to understand the breakdown behaviour of the MGABD functional under different choices of $\psi$ functions and choice of hyperparameter $\beta$ (restricting $\alpha$ to be $1$). In the first scenario, we set  $p_0 = 3/4$ and $x_0 = 10$, representing contamination by an outlier of large magnitude residing deep in the distribution's tail. Then, we calculate the relative estimation error (i.e., $\vert \hat{p} - p_0\vert / p_0$) for various $\psi$ functions, where $\hat{p} = \hat{T}(g_\epsilon)$ denotes the MGABD functional applied to the contaminated mass function $g_\epsilon$. The top panel of Figure~\ref{fig:gab-geometric} illustrates this relative error as a function of the contaminating proportion $\epsilon \in [0, 0.5]$. As evident from the plot, concave choices of $\psi$ (e.g., log) yield superior estimates, and the estimate's tolerance to the contamination proportion increases as the robustness parameter $\beta$ grows.

A contrasting view emerges when we consider the scenario with $p_0 = 1/10$ and $x_0 = 0$, which is described in the bottom panel of Figure~\ref{fig:gab-geometric}. In this regime, the contamination inflates the probability mass at $x = 0$, an event that is otherwise infrequent under the true low success probability. Here, concave $\psi$ function shows degrading performance; instead, convex generators like $\psi(x) = x^2$ deliver relatively greater stability. Naturally, when $\beta = 0$ and $\alpha = 1$, the MGABDE is equivalent to the non-robust MLE regardless of the chosen $\psi$ function. An interesting insight from the relative error curves given in Figure~\ref{fig:gab-geometric} is that the performances of MGABD functional for all evaluated $\psi$ functions are bounded between the extreme cases of the strictly concave $\psi(x) = \ln(x)$ and the strictly convex $\psi(x) = x^2$. This behavior empirically reflects the minimax optimality of the $(\phi,\gamma)$-divergence within the broader GAB class, as established by~\cite{roy2026gaboptimal}. The results by~\cite{roy2026gaboptimal} establish that to find statistically most efficient and robust estimators, one can confine the attention within a slight extension of the J2-divergence family introduced by~\cite{jones2001comparison}, i.e., consider transformations of the form $\psi(x) = \phi^{-1}x^\phi$ for $\phi \geq 0$. Note that, the concave logarithmic transformation ($\phi \to 0+$), identity transformation ($\phi = 1$), and convex quadratic transformation ($\phi = 2$) are members of this family.

\begin{figure}[htbp]
    \includegraphics[width = \textwidth]{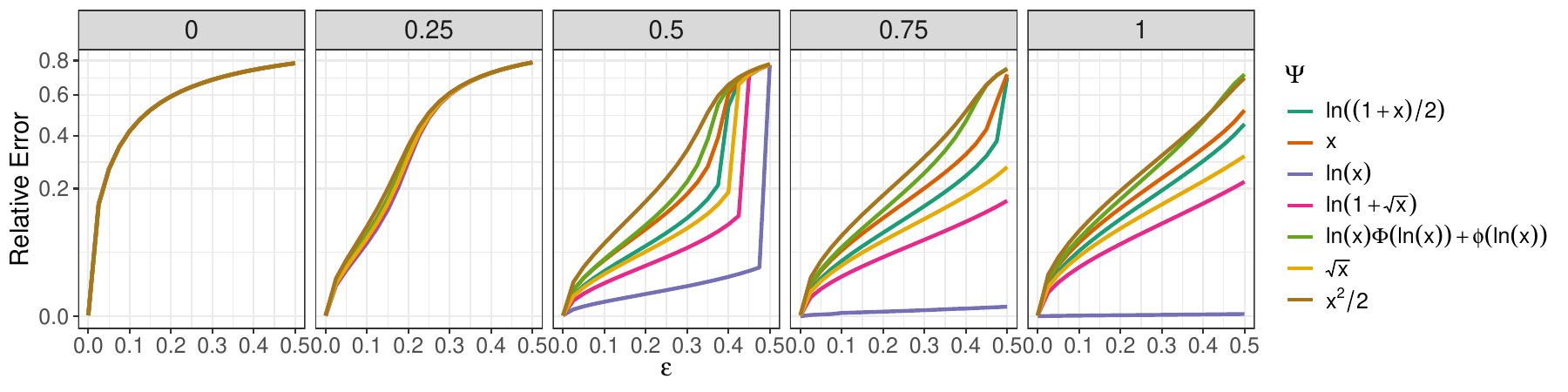}
    \includegraphics[width = \textwidth]{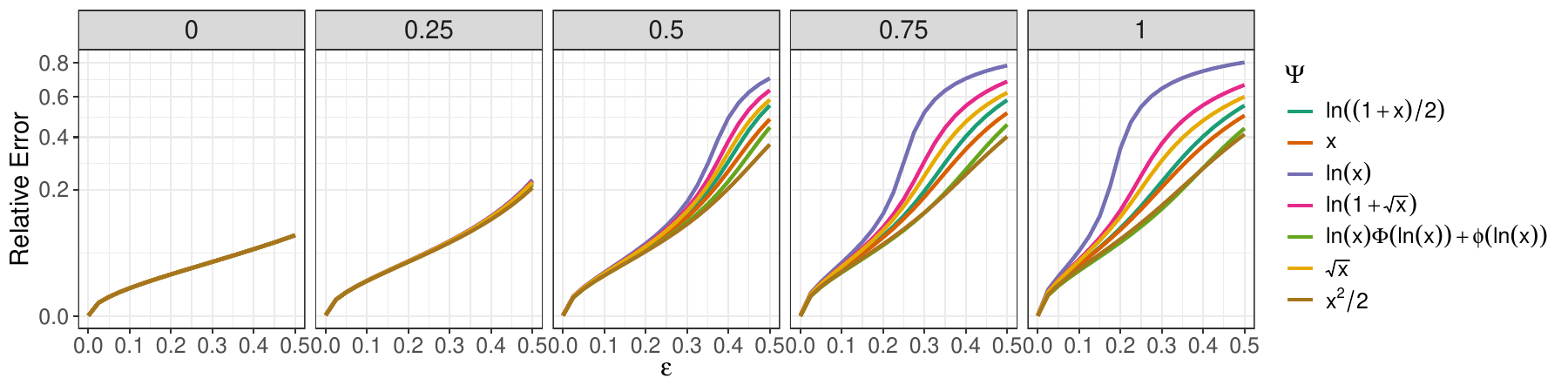}
    \caption{Relative error in estimating success probability of a Geometric distribution for different choices of $\psi$ function and choices of $\beta$ (panel title) and $\alpha = 1$, corresponding to situations: (Top) True $p_0 = 3/4$, contamination at $x_0 = 10$. (Bottom) True $p_0 = 1/10$, contamination at $x_0 = 0$.}
    \label{fig:gab-geometric}
\end{figure}

\begin{figure}[htbp]
    \includegraphics[width = \textwidth]{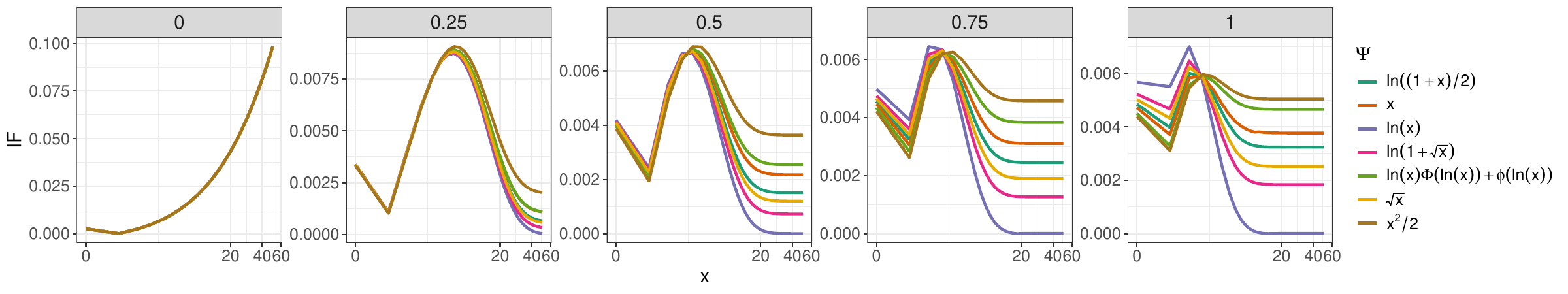}
    \caption{The finite-difference approximation of the influence function of the MGABD functional for different choices of $\psi$ functions and hyperparameters $\alpha = 1$ and varying $\beta$ (panel title), for the problem of estimating success probability in Geometric distribution. (x-axis is in log scale).}
    \label{fig:gab-geom-influence}
\end{figure}

While breakdown behavior focuses on the worst-case scenario, a clearer picture can be observed by examining the influence function of the MGABD functional. To evaluate this numerically, we consider a finite-difference approximation of the influence function given by
\begin{equation}
    \text{IF}(x) = \frac{1}{\epsilon}\left[ \hat{T}((1-\epsilon) g + \epsilon \delta_{x}) - \hat{T}(g) \right],
    \label{eqn:finite-diff-influence}
\end{equation}
\noindent where $g$ represents the true data-generating density (or mass functions) and $\hat{T}$ is the MGABD functional. We choose $\epsilon = 0.01$ in~\eqref{eqn:finite-diff-influence}, and fix the hyperparameters at $\alpha = 1$ and varying levels of $\beta \in [0, 1]$, a regime that generally ensures numerical stability. This can be seen as a population-level equivalent of Tukey's sensitivity curve~\citep{huber2002john}. In this setting, we take the true success probability $p_0 = 1/2$, varying the contamination location $x \in \{0, \dots, 50\}$. Figure~\ref{fig:gab-geom-influence} showcases these finite-difference approximation of the influence function as in~\eqref{eqn:finite-diff-influence}. As illustrated in Figure~\ref{fig:gab-geom-influence}, all evaluated choices of the $\psi$ function lead to redescending influence functions whenever $\beta > 0$. Specifically, the concave generator $\psi(x) = \ln(x)$ aggressively mitigates the impact of large, deep-tail outliers. Conversely, the convex generator $\psi(x) = x^2/2$ provides superior robustness when the outlying observations are concentrated near $x = 0$. Drawing on the intuition developed in Section~\ref{sec:min-gab-est}, this zero-inflation manifests as a model misspecification, a scenario where convex $\psi$ functions deliver reliable inference. Furthermore, while convex $\psi$ functions may incur a slight bias in the presence of large outliers, they offer the benefits of increased statistical efficiency under the true model; formal proofs of this behavior are outlined in~\cite{roy2026gaboptimal}.

\subsection{Illustrations under a continuous model: Normal scale estimation}\label{sec:normal-scale}

\begin{figure}[htbp]
    \includegraphics[width = \textwidth]{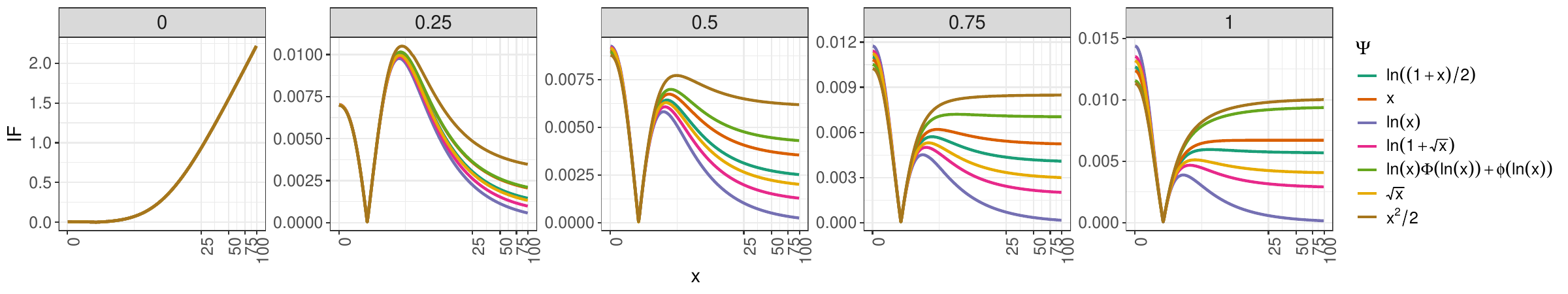}
    \includegraphics[width = \textwidth]{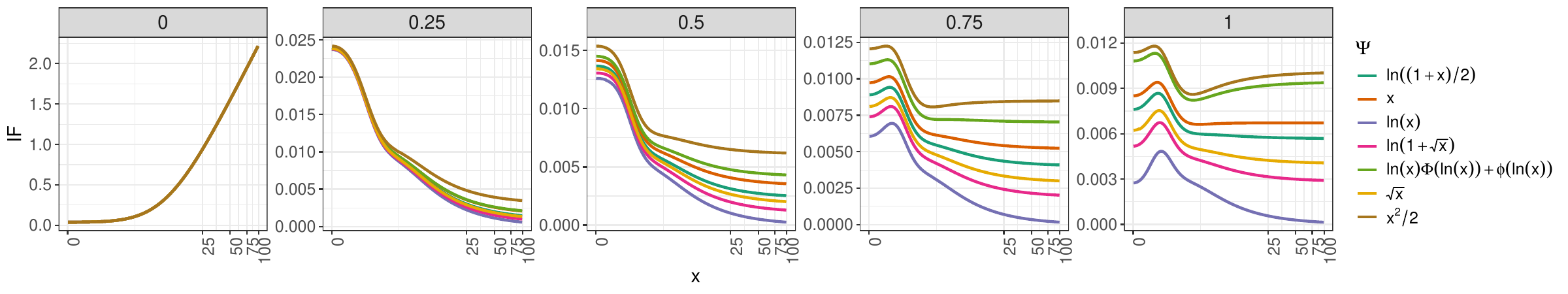}
    \includegraphics[width = \textwidth]{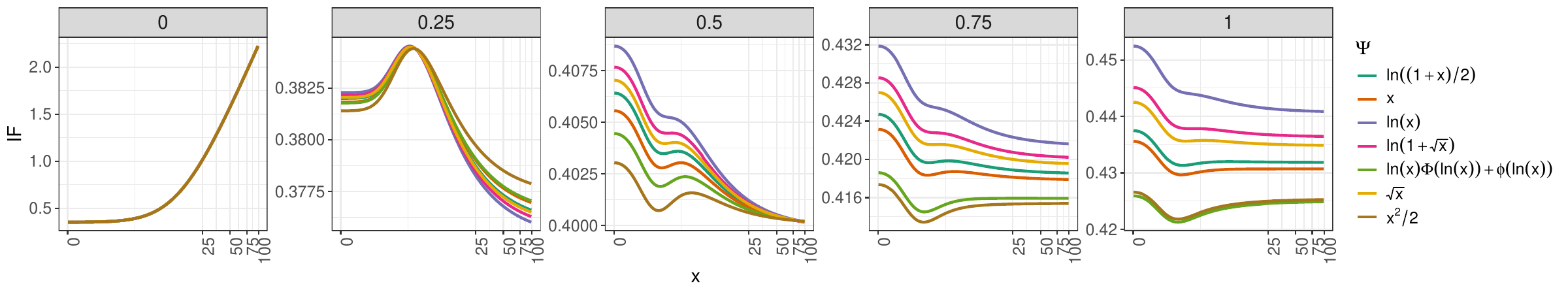}
    \caption{The finite-difference approximation of the influence function of the MGABD functional for different choices $\psi$ functions and hyperparameters $\alpha = 1$ and varying $\beta$ (panel title), for the problem of scale estimation in a normal model, with mean of contaminating distribution $\mu_c$ and mean of model family as $\mu$. (Top) $\mu_c = \mu = 0$. (Middle) $\mu_c = 3, \mu = 0$. (Bottom) $\mu_c = 3, \mu = 1$. (x-axis is in logarithmic scale).}
    \label{fig:gab-normal-influence}
\end{figure}

A careful consideration of the form of GAB divergence reveals that any choice of valid $\psi$-function leads to the same minimum GAB divergence estimator for any location estimation problem. Thus, avoiding this uninteresting case, we consider the standard problem of scale estimation in a normal location-scale family instead. Assume that the true distribution is a standard normal distribution (mean true mean $0$ and true variance $1$), the contaminating distribution is $N(\mu_c, \sigma_c^2)$ and the model family is given by $N(\mu, \sigma^2)$ with a known $\mu$. We consider the following three scenarios.
\begin{enumerate}
    \item[(a)] $\mu_c = \mu = 0$. The contaminating density, the true density and the model density all have known and equal mean.
    \item[(b)] $\mu_c = 3, \mu = 0$. The contaminating density has nonzero mean, but the true mean is known which is used to define the mean of the model family. This is typically an easier problem compared to (a) as the contaminating density has less overlap with the true density.
    \item[(c)] $\mu_c = 3, \mu = 1$. The contaminating density has nonzero mean, and the true mean is not known and one uses a misspecified model family. This is often the most realistic situation and harder than (a).
\end{enumerate}

For each scenario, we compute the finite-difference influence function of the MGABD functional as in~\eqref{eqn:finite-diff-influence} as a function of the contamination scale $\sigma_c = x$, tracking the sensitivity of the estimator under a small contamination proportion of $\epsilon = 0.01$. In Scenario (a), we observe a robustness pattern analogous to the Geometric distribution example in Section~\ref{sec:empirical-geom}. Specifically, convex $\psi$ functions yield greater stability for moderate values of $\sigma_c$, whereas larger values of $\sigma_c$ representing extreme heavy-tailed outliers demand concave $\psi$ generators to ensure reliable estimates. Conversely, in the well-separated contamination regime of Scenario (b), concave choices prove uniformly superior at aggressively downweighting and disregarding gross outliers. In contrast, in the highly challenging regime of Scenario (c), concave $\psi(x) = \ln(x)$ is typically the worst, convex $\psi$ functions uniformly dominating the rest of the choices.

\section{Conclusion}\label{sec:conclusion}

The Generalized Alpha-Beta (GAB) divergence introduced in this paper establishes a unifying framework for deriving novel statistical divergence measures, with promising applications across robust statistics, signal processing, machine learning and pattern recognition. Beyond defining the family, we have derived its necessary and sufficient characterizing properties, demonstrating that fundamental symmetries, scaling behaviors, and projection geometries are shared uniformly across this overarching class. Furthermore, we formulated the associated GAB entropy measure, highlighting its structural utilities within information theory and maximum entropy principles.

Empirically, we have demonstrated that the GAB superfamily offers exceptional flexibility in balancing local outlier resistance against global model reliability. Unlike standard divergence measures that offer a static robustness mechanism, the GAB class facilitates adaptive robust inference that can be governed by suitable adjustments to the curvature of $\psi$. For instance, motivated by the illustrations present in Section~\ref{sec:normal-scale}, one can construct a hybrid $\psi$ function that is concave on $[0, c]$ and convex on $[c, \infty)$ for some threshold $c > 0$, unlocking benefits of both worlds in the scale estimation problem. A large value of $\sigma$ results in a small value of norm $\norm{f_{\theta}}_{\alpha+\beta}$ and the inner products $\inner{f_\theta, g}_{\alpha,\beta}$, which evaluates within the concave domain of $\psi$, in turn, guaranteeing reliable performance under heavy variance inflation regime. Conversely, a small contaminating $\sigma$ evaluates to the convex domain of $\psi$, which activates the global robustifying property of the estimator. Whether such constructions lead to desired inferential properties without incurring extensive computational burden warrants further investigation, and will be taken up in future work.

Additionally, in a sequel paper~\citep{roy2026gaboptimal}, we derive the asymptotic properties of the minimum GAB divergence estimators, including consistency, asymptotic normality, influence function, and breakdown point behaviors. In fact, we characterize the minimax-optimal class of $\psi$ that functions which achieve the highest possible efficiency while retaining an exogenous level of breakdown point guarantee. Furthermore, exploring the conditions under which the symmetrized GAB divergence satisfies the triangle inequality for metrization, establishing generalized information-geometric projection theorems, or practical insights into data-driven tuning of hyperparameters $\alpha$ and $\beta$ as in~\cite{basak2021optimal}, remain exciting avenues for future extensions of this work.

\bibliography{references}

\pagebreak
\appendix

\section{Technical Results and Proofs of the Results}\label{appendix:proofs}

\subsection{Geometric Convexity Lemma}\label{appendix:geom-convex-lemma}

In this section, we present a lemma about the equivalence of the convexity of the function $\Psi(x)$ and the geometric convexity of the $\psi$ function, which we will use repeatedly in the characterization results later on.

\begin{lemma}\label{lem:geometric-convexity}
    Let $\Psi(x) := \psi(e^x)$ for any $x \in \R$ and $\lambda \in [0, 1]$ be a nonnegative real number. Then $\Psi$ is $\lambda$-convex if and only if for any $x, y \geq 0$, $\lambda \psi(x) + (1 - \lambda) \psi(y) \geq \psi(x^\lambda y^{1-\lambda})$.
\end{lemma}

\begin{proof}
    Let us assume that $\Psi$ is convex. Then,
    \begin{align}
        \lambda \psi(x) + (1 - \lambda) \psi(y)
         & = \lambda \Psi( \ln x) + (1 - \lambda) \Psi(\ln y) \nonumber                               \\
         & \geq \Psi\left( \lambda \ln(x) + (1-\lambda) \ln(y) \right) \label{eqn:geomconvex-proof-1} \\
         & = \psi(x^\lambda y^{1-\lambda}), \nonumber
    \end{align}
    \noindent where we make use of the convexity of $\Psi$ in~\eqref{eqn:geomconvex-proof-1}. Conversely, if we assume geometric convexity for $\psi$, then for any $x, y \in \R$ and $\lambda \in [0, 1]$, we have
    \begin{align}
        \lambda \Psi(x) + (1 - \lambda) \Psi(y)
         & = \lambda \psi(e^x) + (1 - \lambda) \psi(e^y)\nonumber                     \\
         & \geq \psi((e^x)^\lambda (e^y)^{1 - \lambda})\label{eqn:geomconvex-proof-2} \\
         & = \psi(e^{\lambda x + (1 - \lambda)y})\nonumber                            \\
         & = \Psi(\lambda x + (1 - \lambda)y),\nonumber
    \end{align}
    \noindent where in~\eqref{eqn:geomconvex-proof-2}, we use the geometric convexity of $\psi$.
\end{proof}

\subsection{Proof of Theorem~\ref{thm:gab-div-alpha-plus-beta-1}}

The cases with $\alpha = 1, \beta = 0$, $\beta = 1, \alpha = 0$ and $\alpha = \beta = 0$ is immediate. Therefore, we focus on the case with $\alpha \notin \{ 0, 1\}$. Now, because the Cressie-Reed power divergence is a valid divergence, it follows that for any $\alpha  \notin \{0, 1\}$ and any pair of sub-densities $p$ and $q$,
\begin{equation*}
    \dfrac{1}{\alpha(1-\alpha)} > \dfrac{1}{\alpha(1-\alpha)} \inner{p,q}_{\alpha,1-\alpha}.
\end{equation*}
\noindent We will consider two separate cases now.

\noindent\textbf{Case 1 with $\alpha \in (0, 1)$:} Assume that $\psi$ is increasing at $1$. By the form of power divergence, it follows that for any pair of sub-densities $p$ and $q$, we have $\inner{p, q}_{\alpha,1-\alpha} < 1$. Using the increasing behavior of $\psi$, it follows that $\psi(\inner{p,q}_{\alpha,1-\alpha}) < \psi(1)$. The nonnegativity of the GAB divergence now follows by rearranging this inequality.

Conversely, assume that the GAB divergence is nonnegative. Since $\inner{p, q}_{\alpha,1-\alpha} \in (0, 1]$, it is enough to show that for any $x \in (0, 1)$, $\psi(x) < \psi(1)$. Consider the densities $p(x) = (\theta+1)^{-1}\ind{[0, \theta+1]}(x)$ and $q(x) = (\theta+1)^{-1}\ind{[1, \theta+2]}(x)$ defined on the alphabet $\Acal = \R$. The nonnegativity of the GAB divergence between these densities $p$ and $q$ reduces to the condition $\psi(\theta/(\theta + 1)) < \psi(1)$. Letting $\theta = x/(1-x)$ now yields $\psi(x) < \psi(1)$ as we intended.

\noindent\textbf{Case 2 with either $\alpha \notin (0, 1)$:} Assume that $\psi$ is increasing at $1$. Similar to the previous case, it follows that for any pair of sub-densities $p$ and $q$, we have $\inner{p, q}_{\alpha,1-\alpha} > 1$, by the nonnegativity of the power divergence. Since $\psi$ is increasing at $1$, we have $\psi(\inner{p,q}_{\alpha,1-\alpha}) > \psi(1)$.

Conversely, assume that the GAB divergence is nonnegative. As before, it is enough to show that for any $x > 1$, we have $\psi(x) > \psi(1)$. Let $p$ and $q$ be the density functions for Gaussian distributions with mean $0$ and $\theta$ respectively but with a common variance equal to $1$. It is easy to see that $\inner{p, q}_{\alpha,1-\alpha} = \exp(-\theta^2\alpha(1-\alpha)/2)$. As $\alpha \notin [0,1]$, we must have $\alpha(1-\alpha) < 0$. If we pick $\theta = \sqrt{-2x/(\alpha(1-\alpha))}$, then the nonnegativity of the GAB divergence leads to the inequality $\psi(1) \leq \psi(x)$ as we wanted.

\subsection{Proof of Theorem~\ref{thm:gab-div-nec-suff}}\label{sec:appendix-proof-gab-nec-suff}

We show this proof for the necessary and sufficient conditions separately, by means of the following two lemmas.

\begin{lemma}\label{lem:gab-sufficient}
    Define $\Psi(x) := \psi(e^x)$ for all $x \in [0, \infty)$. If $\Psi$ is strictly increasing and convex, then the GAB generated by $\psi$ is nonnegative for all choices of $\alpha$ and $\beta$ such that at least one of $\alpha$ and $\beta$ is nonzero and $(\alpha + \beta) \neq 0$.
\end{lemma}

\begin{proof}
    If $P = Q$, then it is straightforward to verify that the GAB divergence is equal to $0$. Therefore, we focus on the case when $P \neq Q$. Based on the signs of $\alpha\beta,\alpha(\alpha+\beta)$ and $\beta(\alpha+\beta)$, we deal with different cases separately using different versions of H\"{o}lder's inequality.

    \noindent\textbf{Case 1 with $\alpha\beta > 0, \alpha(\alpha+\beta) > 0, \beta(\alpha+\beta) > 0$:} By applying H\"{o}lder's inequality on $p^{\alpha+\beta}$ and $q^{\alpha+\beta}$, we have
    \begin{equation}
        \int p^\alpha q^\beta \leq \left( \int p^{\alpha+\beta} \right)^{\lambda_\alpha} \left( \int q^{\alpha+\beta} \right)^{\lambda_\beta},
        \label{eqn:thm-gab-sufficient-1}
    \end{equation}
    \noindent where $\lambda_\alpha = \alpha/(\alpha+\beta)$ and $\lambda_\beta = (1 - \lambda_\alpha)$. Since $\Psi$ is convex, by applying Lemma~\ref{lem:geometric-convexity}, we have
    \begin{equation*}
        \lambda_\alpha \psi\left( \norm{p}_{\alpha+\beta}^{\alpha+\beta}\right) + \frac{\beta}{\alpha+\beta} \psi\left( \norm{q}_{\alpha+\beta}^{\alpha+\beta}\right)
        \geq  \psi\left( \norm{p}_{\alpha+\beta}^\alpha \norm{q}_{\alpha+\beta}^\beta \right)
        \geq \psi\left( \inner{p,q}_{\alpha,\beta} \right),
    \end{equation*}
    \noindent where the last inequality follows from the inequality~\eqref{eqn:thm-gab-sufficient-1} and the strictly increasing nature of $\psi$. Dividing both sides by $\alpha\beta$ now shows that the GAB divergence form as in~\eqref{eqn:defn-gab-div} is nonnegative. Note that, the equality holds if and only if $p^{\alpha+\beta} = cq^{\alpha+\beta}$ for some constant $c > 0$. Since $p$ and $q$ are densities, then they are equal almost surely.

    \noindent\textbf{Case 2 with $\alpha\beta < 0, \alpha(\alpha+\beta) > 0, \beta(\alpha+\beta) < 0$:} Note that, $(\alpha + \beta)/\alpha > 0$, $-\beta/\alpha > 0$ and they add up to $1$. Therefore, by applying H\"{o}lder's inequality on $p^{\alpha}q^\beta$ and $q^{\alpha+\beta}$, we get
    \begin{equation*}
        \left( \int p^{\alpha} q^\beta \right)^{(\alpha+\beta)/\alpha} \left( \int q^{\alpha+\beta} \right)^{-\beta/\alpha} \geq \int p^{\alpha+\beta}.
    \end{equation*}
    \noindent Using the increasingness and geometric-convexity of $\psi$, we obtain the chain of inequalities
    \begin{equation*}
        \dfrac{\alpha+\beta}{\alpha} \psi\left( \inner{p,q}_{\alpha,\beta} \right) - \dfrac{\beta}{\alpha} \psi\left( \norm{q}_{\alpha+\beta}^{\alpha+\beta} \right)
        \geq \psi\left( \inner{p, q}_{\alpha,\beta}^{(\alpha+\beta)/\alpha} \norm{q}_{\alpha+\beta}^{-\beta(\alpha+\beta)/\alpha} \right)
        \geq \psi\left( \int p^{\alpha+\beta} \right).
    \end{equation*}

    \noindent Dividing both sides by $\beta(\alpha+\beta)$ yields,
    \begin{equation*}
        \dfrac{1}{\alpha\beta} \psi\left( \inner{p,q}_{\alpha,\beta} \right) - \dfrac{\lambda_\beta}{\alpha\beta}\psi\left( \norm{q}_{\alpha+\beta}^{\alpha+\beta} \right) \leq \dfrac{\lambda_\alpha}{\alpha\beta} \psi\left( \norm{p}_{\alpha+\beta}^{\alpha+\beta} \right).
    \end{equation*}
    \noindent A rearrangement of the above shows that the GAB divergence is nonnegative. Similar to the previous case, it is easy to see that if $p = q$, then GAB divergence is equal to $0$.

    \noindent\textbf{Case 3 with $\alpha\beta < 0, \alpha(\alpha+\beta) < 0, \beta(\alpha+\beta) > 0$:} This follows from the previous case and the duality of the GAB as in~\eqref{eqn:gab-symmetry}.

    \noindent\textbf{Case 4 with $\alpha \neq 0, \beta = 0$:} Let $p$ and $q$ be two sub-densities. Let $x = \ln(\norm{p}_{\alpha}^{\alpha})$ and $y = \ln(\norm{q}_{\alpha}^{\alpha})$, then the GAB divergence given in~\eqref{eqn:gab-to-kl-bzero} reduces to
    \begin{equation*}
        \alpha^2 \dGAB^{(\alpha,0),\psi}
        = \Psi'(x) d^\ast(P^{[\alpha]}, Q^{[\alpha]}) + \left( \Psi(y) - \Psi(x) - \Psi'(x)(y-x) \right).
    \end{equation*}
    \noindent Since, $\Psi'(x) > 0$ due to the increasing nature of $\Psi$ and the KL-divergence is nonnegative, it is enough to show that $\Psi(y)-\Psi(x)-\Psi'(x)(y-x)$ is nonnegative. It now follows from the fact that $\Psi$ is convex and $\Psi'(x)$ is a subgradient of $\Psi$ at $x$. In particular, there can be three cases. If $y = x$, then the right-hand side is $0$, hence nonnegative. If $y > x$, then by exploiting the convexity and increasing property of $\Psi$, we get $\Psi'(x)\le \Psi'(y)$ and $\Psi(x)\le \Psi(y)$. So, by Lagrange's Mean Value theorem, there exists $z\in (x,y)$ such that
    \begin{equation*}
        \Psi'(x) \leq \Psi'(z) = \frac{\Psi(y)-\Psi(x)}{y-x}\le \Psi'(y).
    \end{equation*}
    \noindent On the other hand if $y < x$, then a similar deduction applies. This completes the proof for this case.

    \noindent\textbf{Case 5 with $\alpha = 0, \beta \neq 0$:} The case for $\alpha = 0, \beta \neq 0$ is straightforward due to the symmetry of GAB divergence as in Eq.~\eqref{eqn:gab-symmetry} and the previous case.
\end{proof}

\begin{lemma}\label{lem:gab-necessary}
    If the GAB divergence with a generating function $\psi \in C^1([0, \infty))$ is nonnegative, then $\Psi(x) := \psi(e^x)$ must be strictly increasing and convex, for any choice of $\alpha$ and $\beta$ such that
    at least one of $\alpha$ and $\beta$ is nonzero and $(\alpha + \beta) \notin \{0, 1\}$.
\end{lemma}

\begin{proof}
    To establish the necessary condition, we proceed by carefully constructing a pair of nonnegative measures such that the nonnegativity of the GAB divergence between them yields the necessary increasing and convexity properties. Similar to the proof of Proposition~\ref{lem:gab-sufficient}, we split the proof by considering three scenarios separately.

    \noindent\textbf{Case 1 with $\alpha\beta > 0, \alpha(\alpha+\beta) > 0$ and $\beta(\alpha + \beta) > 0$:}

    \noindent\emph{Increasingness:} To show that $\Psi$ is strictly increasing, it is enough to show that $\psi$ is strictly increasing. Take any $y > z > 0$. Consider two sub-densities $p(x) = c\ind{[0, y]}(x)$ and $q(x) = c\ind{[y-z, 2y-z]}(x)$ where $c$ is a sufficiently small fixed quantity. Note that, $\norm{p}_{\alpha+\beta}^{\alpha+\beta} = \norm{q}_{\alpha+\beta}^{\alpha+\beta} = c^{\alpha+\beta} y$ and $\inner{p,q}_{\alpha,\beta} = c^{\alpha+\beta}(y-(y-z)) = c^{\alpha+\beta} z$. From the nonnegativity of the GAB divergence and denoting $c' = c^{\alpha+\beta}$, we get
    \begin{equation*}
        \dfrac{1}{\beta(\alpha+\beta)} \psi(c' y) + \dfrac{1}{\alpha(\alpha+\beta)}\psi(c' y) - \dfrac{1}{\alpha\beta}\psi(c' z) > 0,
    \end{equation*}
    \noindent which can be rearranged as $\psi(c^{\alpha+\beta} y) > \psi(c^{\alpha+\beta} z)$. Since $y > z$ is arbitrary, this establishes the increasing nature of $\psi$.

    \noindent\emph{Convexity:} To show the convexity, consider again any $y > z > 0$. Let us consider the pair of sub-densities $p$ and $q$ given by
    \begin{align}
        p(x) & = c(\gamma(\alpha+\beta) + 1)^{1/(\alpha+\beta)}\theta^{-(\gamma+1)}x^{\gamma}\bb{1}_{(0,\theta)}(x),\nonumber \\
        q(x) & = c(\gamma(\alpha+\beta) + 1)^{1/(\alpha+\beta)}\eta^{-(\gamma+1)}x^{\gamma}\bb{1}_{(0,\eta)}(x),
        \label{eqn:f-g-power-family}
    \end{align}
    \noindent where $c$ is a small constant. Let $c' = c^{\alpha+\beta}$. Then it follows that
    \begin{equation*}
        \norm{p}_{\alpha+\beta}^{\alpha+\beta} =  c'\theta^{-(\alpha+\beta)+1},
        \
        \norm{q}_{\alpha+\beta}^{\alpha+\beta} = c' \eta^{-(\alpha+\beta)+1},
    \end{equation*}
    \noindent and,
    \begin{equation*}
        \inner{p,q}_{\alpha,\beta} = \begin{cases}
            c'\theta^{-(\gamma+1)\beta + (1-\alpha-\beta)} \eta^{-(\gamma+1)\beta} & \text{ if } \theta \leq \eta, \\
            c'\theta^{-(\gamma+1)\alpha} \eta^{(\gamma+1)\alpha+(1-\alpha-\beta)}  & \text{ if } \theta > \eta,
        \end{cases}
    \end{equation*}
    \noindent where $\theta, \eta, \gamma$ are parameters to be suitably chosen. We take $\theta = e^{-y/(\alpha+\beta - 1)}, \eta = e^{-z/(\alpha+\beta-1)}$.

    If $\alpha + \beta > 1$, then $y > z$ implies that $\theta < \eta$. The nonnegativity of the GAB divergence then reduces to
    \begin{equation*}
        \dfrac{1}{\beta(\alpha+\beta)}\Psi(c' y) + \dfrac{1}{\alpha(\alpha+\beta)}\Psi(c' z)
        > \dfrac{1}{\alpha\beta} \Psi\left( \left( 1 - \frac{\beta(\gamma+1)}{\alpha+\beta - 1} \right)c'y + \frac{\beta(\gamma+1)}{(\alpha+\beta-1)}c' z \right).
    \end{equation*}
    \noindent Letting $\gamma = -1/(\alpha+\beta)$ and multiplying both sides by $\alpha\beta$ we get that $\Psi$ is $\alpha/(\alpha+\beta)$-convex on $\R$. Since $\Psi$ is continuous, an application of Lemma 1.2 of Appendix A of~\citep{ray2023functional} now shows that $\Psi$ is convex.

    If $\alpha + \beta < 1$, then we still have $\beta/\alpha > 0$ and so $(1 + \beta/\alpha) > 1$. Due to the relationship~\eqref{eqn:gab-reduction-to-1}, we can modify the hyperparameters as $\alpha' = 1$ and $\beta' = \beta/\alpha$, which reduces to the former case with $\alpha + \beta > 1$.

    \noindent\textbf{Case 2 with $\alpha\beta < 0, \alpha(\alpha+\beta) > 0$ and $\beta(\alpha+\beta) < 0$:}

    \noindent\emph{Increasingness:} Similar to the above, to show the increasing nature of $\Psi$, it is enough to show that for any $y > z > 0$, $\psi(y) > \psi(z)$. To see this, we consider the pair of standard normal densities
    \begin{equation*}
        p(x) = \frac{1}{\sqrt{2\pi}\sigma}e^{-x^2/2\sigma^2}, \
        q(x) = \frac{1}{\sqrt{2\pi}\sigma}e^{-(x-\theta)^2/2\sigma^2}.
    \end{equation*}
    \noindent Elementary calculations yield that with $C = (2\pi)^{-(\alpha+\beta - 1)/2}\sigma^{-(\alpha+\beta - 1)}(\alpha+\beta)^{-1/2}$,
    \begin{equation*}
        \norm{p}_{\alpha+\beta}^{\alpha+\beta}
        = \norm{q}_{\alpha+\beta}^{\alpha+\beta} = C, \
        \inner{p,q}_{\alpha,\beta} = Ce^{-\frac{\alpha\beta\theta^2}{2(\alpha+\beta)\sigma^2} }.
    \end{equation*}

    If $(\alpha + \beta) \in (0, 1)$, then we choose
    \begin{align*}
        \sigma & = (2\pi)^{-1/2}(\alpha+\beta)^{-1/2(\alpha+\beta - 1)}z^{-1/(\alpha+\beta-1)}, \\
        \theta & = \sqrt{-2(\alpha+\beta)/\alpha\beta} (\ln(y/z))^{1/2}\sigma.
    \end{align*}
    \noindent With this specific choice, the nonnegativity of the GAB divergence leads to the inequality
    \begin{equation*}
        \dfrac{1}{\beta(\alpha+\beta)}\psi(z) - \dfrac{1}{\alpha\beta}\psi(y) + \dfrac{1}{\alpha(\alpha+\beta)}\psi(z) > 0.
    \end{equation*}
    \noindent Multiplying both sides by $\alpha\beta$ and rearranging above ensures $\psi(y) > \psi(z)$, as we wanted.

    If $\alpha + \beta < 0$ or $\alpha + \beta > 1$, then note that by choice of the case, we have $\beta/\alpha < 0$ and $(\alpha+\beta)/\alpha = (1 + \beta/\alpha) \in (0, 1)$. Then an application of the relationship~\eqref{eqn:gab-reduction-to-1} again reduces it to the former case $(\alpha+\beta) \in (0, 1)$, with new hyperparameters $\alpha' = 1$ and $\beta' = \beta/\alpha$.

    \noindent\emph{Convexity:} To show that $\Psi$ is convex, without the loss of generality, let us consider two real numbers $y > z \geq 0$. We consider the pair of nonnegative sub-densities given in~\eqref{eqn:f-g-power-family} again. Since $\alpha\beta < 0$, at least one of $\alpha$ and $\beta$ must be positive. Assume $\alpha > 0$ without the loss of generality. We take $\eta = e^{-y/(\alpha+\beta-1)}$ and take $\theta = \eta e^{(y-z)/(\gamma+1)\alpha}$. For this specific choice, the nonnegativity of the GAB divergence yields
    \begin{equation*}
        -\dfrac{\beta}{\alpha}\Psi\left(c' y \right) + \left( \dfrac{\alpha+\beta}{\alpha} \right) \Psi(c' z)
        \geq \Psi\left( \left( 1 - \dfrac{\alpha+\beta-1}{(\gamma+1)\alpha} \right)c' y + \dfrac{\alpha+\beta-1}{(\gamma+1)\alpha}c' z \right).
    \end{equation*}
    \noindent If $\alpha+\beta > 1$, then taking $\gamma = -1/(\alpha+\beta)$ shows that $\Psi$ is $(-\beta/\alpha)$-convex on $\R$. The rest of the proof now follows by continuity of $\Psi$ and an application of Lemma 1.2 of~\citep{ray2023functional}.

    If $\alpha + \beta \in (0, 1)$, there exists a sufficiently large but fixed $M > 0$ such that $M(\alpha+\beta) > 1$. Now, due to the relationship~\eqref{eqn:gab-zoom}, considering the new hyperparameters $\alpha' = M\alpha$ and $\beta' = M\beta$, and starting with the nonnegativity of the GAB divergence between $P^{1/M}$ and $Q^{1/M}$ leads to the former case with $\alpha + \beta > 1$.

    If $\alpha + \beta < 0$, then we note that as $\alpha(\alpha+\beta) > 0$, both $\alpha$ and $(\alpha+\beta)$ are of the same sign. Considering the new hyperparameters $\alpha' = 1$ and $\beta' = \beta/\alpha$ yields that $(\alpha'+\beta') = (\alpha+\beta)/\alpha > 0$. Therefore, using the relationship~\eqref{eqn:gab-reduction-to-1} we can again reduce it to one of the above cases.

    \noindent\textbf{Case 3 with $\alpha\beta < 0, \alpha(\alpha+\beta) < 0$ and $\beta(\alpha+\beta) > 0$:} This case follows from the previous case and is a direct application of the duality property of GAB divergence as shown in~\eqref{eqn:gab-symmetry}.

    \noindent\textbf{Case 4 with $\alpha \notin \{0, 1\}, \beta = 0$:}

    \noindent\emph{Increasingness:} Since $\psi \in C^1(\R)$, to show that $\Psi$ is strictly increasing, it is enough to show that for any $y > 0$, $\psi'(y) > 0$. Consider the pair of Gaussian densities $p$ and $q$,
    \begin{equation*}
        p(x) = \dfrac{1}{\sqrt{2\pi} \sigma} e^{-x^2/2\sigma^2}, \
        \text{and, }
        q(x) = \dfrac{1}{\sqrt{2\pi}\sigma}e^{-(x - \theta)^2/2\sigma^2}.
    \end{equation*}
    \noindent It follows that $\norm{p}_{\alpha}^{\alpha} = \norm{q}_{\alpha}^{\alpha} = (2\pi)^{-(\alpha-1)/2}\sigma^{-(\alpha-1)}\alpha^{-1/2}$. As $\alpha \neq 1$, we take
    \begin{equation*}
        \sigma = (2\pi)^{-1/2}(\alpha)^{-1/2(\alpha-1)} y^{-1/(\alpha-1)} > 0.
    \end{equation*}
    \noindent The nonnegativity of the GAB divergence as in~\eqref{eqn:gab-to-kl-bzero} now reduces to
    \begin{equation*}
        \psi'(y)y d^\ast(P^{[\alpha]}, Q^{[\alpha]}) - \psi(y) + \psi(y) > 0,
    \end{equation*}
    \noindent and since the KL-divergence is nonnegative, it implies that $\psi'(y) > 0$, as we wanted to show.

    \noindent\emph{Convexity:} Moving on to the convexity condition, let us pick any $y > z \geq 0$. We consider again the family of nonnegative sub-densities $p$ and $q$ given in~\eqref{eqn:f-g-power-family}. If $\alpha > 1$, then we pick $\theta = e^{-y/(\alpha-1)}$ and $\eta = e^{-z/(\alpha-1)}$. As $\theta < \eta$, elementary calculations show that
    \begin{equation*}
        \int p^\alpha \ln(p/q)
        = c^\alpha (\gamma+1)(\ln(\theta) - \ln(\eta)) \theta^{-(\alpha-1)}
        = (\gamma+1)e^y (z - y)/(\alpha-1).
    \end{equation*}
    \noindent Therefore, the nonnegativity of the GAB divergence translates to the inequality
    \begin{equation*}
        \dfrac{(\gamma+1)\alpha}{(\alpha-1)}\psi'(e^y)e^y (z - y) - \psi(e^y) + \psi(e^z) \geq 0.
    \end{equation*}
    \noindent Choosing $\gamma = -1/\alpha$ yields $\Psi(y) - \Psi(z) - \Psi'(y)(y - z) \geq 0$ for any $y > z$, which establishes the convexity of $\Psi$ function.

    If $\alpha \in (0, 1)$, then there exists a constant $M > 0$ such that $M\alpha > 1$. By considering the modified parameter $\alpha' = M\alpha$ and utilizing the zooming property as in~\eqref{eqn:gab-zoom}, we reduce this to the former case with $\alpha > 1$.

    If $\alpha < 0$, we can apply the same zooming property~\eqref{eqn:gab-zoom}, but with the zooming constant $w = (-1)$.

    \noindent\textbf{Case 5 with $\alpha = 0, \beta \notin \{0, 1\}$:} As before, the case for $\alpha = 0, \beta \notin \{0, 1\}$ follows from the symmetry of GAB divergence as in Eq.~\eqref{eqn:gab-symmetry} and the previous case.
\end{proof}

\subsection{Proof of Theorem~\ref{thm:gab-div-nec-suff-case4}}

\textbf{Sufficiency part:} We start with the sufficiency part first. Let $r(x) = \ln(p(x) / q(x))$. Then, the form of GAB divergence for $\alpha = -\beta \neq 0$ as given in Eq.~\eqref{eqn:defn-gab-div-edge} reduces to
\begin{equation}
    \dGAB^{(\alpha,-\alpha),\psi}(P, Q) = \dfrac{1}{\alpha^2} \left[ \psi( \textstyle \int e^{\alpha r} ) - \psi(1) - \psi'(1) \int \alpha r \right].
    \label{eqn:gab-div-alphabetazero}
\end{equation}

The nonnegativity of the GAB divergence can be established due to the following chain of inequalities.
\begin{align}
    \psi( \textstyle \int e^{\alpha r} ) - \psi(1) - \psi'(1) \int \alpha r
     & = \Psi( \textstyle \ln \int e^{\alpha r}) - \Psi(0) - \Psi'(0) \int \alpha r \nonumber                         \\
     & > \Psi( \textstyle \int \alpha r) - \Psi(0) - \Psi'(0) \int \alpha r, \label{eqn:gab-div-alphabetazero-suff-1} \\
     & \geq 0, \label{eqn:gab-div-alphabetazero-suff-2}
\end{align}
\noindent Here, the inequality~\eqref{eqn:gab-div-alphabetazero-suff-1} follows from Jensen's inequality that $\int e^{\alpha r} \geq e^{\int \alpha r}$, and the strictly increasing property of $\Psi$. The inequality~\eqref{eqn:gab-div-alphabetazero-suff-2} is a direct consequence of the convexity of $\Psi$, and follows from similar arguments as in Case 4 of the proof of Proposition~\ref{lem:gab-sufficient}.

However, if $\Psi'(0) = 0$, then the GAB divergence reduces to the expression $\Psi(\ln \textstyle \int e^{\alpha r}) - \Psi(0)$. This is nonnegative since $\Psi(x) > \Psi(0)$ for any $x \neq 0$.

\noindent \textbf{Necessary condition:} Proceeding to establish the necessary condition, we show this for two cases separately, depending on whether $\Psi'(0) = 0$ or not.

\noindent \emph{Case A $(\Psi'(0) = 0)$:} First note that due to the reduced form of GAB divergence as in Eq.~\eqref{eqn:gab-div-alphabetazero}, if $\Psi'(0) = 0$, then the nonnegativity of GAB divergence leads to the inequality,
\begin{equation}
    \Psi(\ln \textstyle\int e^{\alpha r}) - \Psi(0) > 0,
    \label{eqn:gab-div-alphabetazero-nec-1}
\end{equation}
\noindent where $r = \ln(p/q)$. Now pick any $x \neq 0$. We choose the densities $p$ and $q$ as
\begin{equation*}
    p(u) = \begin{cases}
        2a\frac{(1-b)}{(a-b)} & \text{ if } u \in [0, 1/2] \\
        2b\frac{(a-1)}{(a-b)} & \text{ if } u \in (1/2, 1]
    \end{cases},
    \ \quad
    q(u) = \begin{cases}
        2\frac{(1-b)}{(a-b)} & \text{ if } u \in [0, 1/2] \\
        2\frac{(a-1)}{(a-b)} & \text{ if } u \in (1/2, 1]
    \end{cases},
\end{equation*}
\noindent where the variables $a, b$ are to be described shortly. Note that, for any choice of $a$ and $b$, the above expressions lead to valid density functions. These careful constructions help us to express the integral
\begin{equation*}
    \int e^{\alpha r(u)}du = \int \left( \frac{p(u)}{q(u)}\right)^\alpha du = \frac{1}{2}\left( e^{\alpha a} + e^{\alpha b} \right).
\end{equation*}
\noindent Given any $x \neq 0$, one can obtain values of $a$ and $b$ such that the above expression equals $e^x$, which leads to the inequality $\Psi(x) > \Psi(0)$ as a consequence of~\eqref{eqn:gab-div-alphabetazero-nec-1}.

\noindent \emph{Case B $(\Psi'(0) \neq 0)$:} We indicate the proof for establishing increasingness and convexity property of $\Psi$ separately.

\noindent\emph{Increasingness:} Given any two real numbers $y > x \geq 0$, our goal is to show that $\Psi(y) > \Psi(x)$. Let us choose
\begin{equation*}
    r(u) = \begin{cases}
        a    & \text{ if } u \in [0, c], \\
        (-b) & \text{ if } u \in (c, 1],
    \end{cases}
\end{equation*}
\noindent where $c \in (0, 1)$ satisfies
\begin{equation*}
    \frac{c}{1-c} = \frac{\sinh(b/2)}{\sinh(a/2)}.
\end{equation*}
\noindent With this specific choice of $r$-function and the choice of $c$, elementary calculations show that one can choose $p(u) \propto e^{r(u)/2}$ and $q(u) \propto e^{-r(u)/2}$ leading to valid density functions. Now, we pick $a$ and $b$ such that they are the solutions to the simultaneous nonlinear equations given by
\begin{align*}
    \int_0^1 r(u)du            & = ac - b(1-c) = \frac{\Psi(x) - \Psi(0)}{\Psi'(0)}, \\
    \int_0^1 e^{\alpha r(u)}du & = ce^{\alpha a} + (1-c) e^{-\alpha b} = e^{y},
\end{align*}
\noindent which is well-defined as $\Psi'(0) \neq 0$. The feasibility of this solution can be established by applications of intermediate value theorems, and considering the behavior of the left-hand sides under appropriate limits. Equipped with this solution, the nonnegativity of the GAB divergence for this choice of $P$ and $Q$ leads to the inequality, $\Psi(y) - \Psi(x) > 0$, as we wanted to show.

\emph{Convexity:} We shall show that for any $x > 0$, the quantity $\Psi(x) - \Psi(0) - \Psi'(0)x > 0$. Without loss of generality, assume that $\alpha > 0$, otherwise, we can simply apply the same arguments with $\beta = -\alpha$. Choose $a = e^{x/\alpha}$. Clearly, $a > 1$. Then, we pick the densities $p$ and $q$ to be uniform densities over the region $[0, 1]$ and $[0, a]$ respectively. It follows that,
\begin{align*}
    \int \alpha r(u)du     & = \int_0^1 \alpha\ln(a)du = \ln(a^\alpha) = x, \\
    \int e^{\alpha r(u)}du & = \int_0^1 e^{\alpha \ln(a)}du = e^x.
\end{align*}
\noindent As a result, the nonnegativity of the GAB divergence as given in Eq.~\eqref{eqn:gab-div-alphabetazero} leads to the desired inequality
\begin{equation*}
    \Psi'(x) - \Psi(0) - \Psi'(0) x > 0,
\end{equation*}
\noindent where strict inequality follows since the densities $p$ and $q$ are not identical.

\subsection{Proof of Proposition~\ref{thm:gab-semicont}}

Suppose that the inequality~\eqref{eqn:liminf-gab-div} does not hold. In that case, we can find a subsequence $\{ n_j\}_{j=1}^\infty$ such that
\begin{equation}
    \lim_{j\rightarrow \infty} \dGAB^{(\alpha,\beta),\psi}(P_{n_j}, Q) < \dGAB^{(\alpha,\beta),\psi}(P, Q).
    \label{eqn:semicont-proof-1}
\end{equation}
\noindent For this subsequence, we have $\int \vert p_{n_j} - p\vert^{\alpha+\beta}d\mu \rightarrow 0$ as $P_n$ converges to $P$ in $L_{\alpha+\beta}(\mu)$ metric. It follows that there is a further subsequence $\{ n_{j_k} \}_{k=1}^\infty$ such that $p_{n_{j_k}}(a) \rightarrow p(a)$ as $k \rightarrow \infty$ for almost all $a \in \Acal$ with respect to $\mu$. Then, by a generalized DCT-type argument (see Lemma A.1 of~\cite{roy2026asymptotic}), it follows that for this specific subsequence, we have $p_{n_{j_k}}^{\alpha+\beta}$ converging to $p^{\alpha+\beta}$ in $L^1(\mu)$. In other words, $p_{n_{j_k}}^{\alpha+\beta}/q^{\alpha+\beta}$ converges to $(p/q)^{\alpha+\beta}$ in $L^1(q^{\alpha+\beta})$.

If $\beta < 0$, we make use of Lemma 1 of~\cite{teboulle1993convergence}, which establishes the semi-continuity of the function $f \mapsto \int h(f)\dd\nu$ for any convex function $h$ and probability measure $f$ and any nonnegative measure $\nu$. We take $h(x) = x^{\alpha/(\alpha+\beta)}$, which is convex as $\beta < 0$, $f = c(p/q)^{\alpha+\beta}$, where $c$ is an appropriate normalizing constant and $\nu = q^{\alpha+\beta}$. As a result, we obtain
\begin{equation*}
    \liminf_{n\rightarrow \infty} \int p_n^{\alpha} q^\beta \dd\mu \geq \int p^\alpha q^\beta \dd\mu.
\end{equation*}
\noindent In particular, considering the semi-continuity for the specific subsequence $\{ n_{j_k} \}_{k=1}^\infty$ in conjunction with the continuity of $\psi$ function, we obtain
\begin{equation}
    -\lim_{k \rightarrow \infty} \dfrac{1}{\alpha\beta} \psi\left( \int p_{n_{j_k}}^\alpha q^\beta \right) \geq -\dfrac{1}{\alpha\beta}\psi\left( \inner{p, q}_{\alpha,\beta} \right),
    \label{eqn:semicont-proof-2}
\end{equation}
\noindent since $\alpha\beta < 0$. As we have already established that $p_{n_{j_k}}^{\alpha+\beta}$ converges to $p^{\alpha+\beta}$ in $L^1(\mu)$ as $k \rightarrow \infty$, in conjunction with the continuity of $\psi$ function, we get
\begin{equation}
    \lim_{k \rightarrow \infty} \dfrac{1}{\beta(\alpha+\beta)} \psi\left( \int p_{n_{j_k}}^{\alpha+\beta} \right) = \dfrac{1}{\beta(\alpha+\beta)} \psi(\norm{p}_{\alpha+\beta}^{\alpha+\beta}).
    \label{eqn:semicont-proof-3}
\end{equation}
\noindent Adding both sides of the inequalities~\eqref{eqn:semicont-proof-2} and~\eqref{eqn:semicont-proof-3} yields that for the specific subsequence $\lim_{k\rightarrow \infty} \dGAB^{(\alpha,\beta),\psi}(p_{n_{j_k}}, q) \geq \dGAB^{(\alpha,\beta),\psi}(p, q)$. This contradicts the inequality~\eqref{eqn:semicont-proof-1}.

When $\beta > 0$, since $\alpha+\beta > \alpha$, the $L_{\alpha+\beta}(\mu)$ convergence of $P_n$ to $P$ implies $L_\alpha(\mu)$ convergence as well. Using Lemma A.1 of~\cite{roy2026asymptotic}, we note that the inequality~\eqref{eqn:semicont-proof-2} holds with equality in this case. As a result, we can establish that the lower semi-continuity follows with an equality in~\eqref{eqn:semicont-proof-1} instead of an inequality. By a similar argument as above, one can show that the GAB divergence, in this case, is also upper semicontinuous. Therefore, it is continuous when $\beta > 0$.



\subsection{Proof of Proposition~\ref{thm:gab-entropy-concave}}

Before proceeding with the proof, we present the definition of quasi-convexity and a subsequent technical lemma that describes the convexity nature of the $\psi(\norm{p}_a^a)$ under different conditions.

\begin{definition}\label{defn:quasi-convex}
    A function $F : \Mfrak_{1}(\Acal) \to \R$ is said to the $\gamma$-quasi-convex if for all $p_0, p_1 \in \Mfrak_1(\Acal)$ and any $\lambda \in [0, 1]$,
    \begin{equation*}
        F(p) \leq \lambda F(p_0) + (1-\lambda)F(p_1),
    \end{equation*}
    \noindent where $p^\gamma = \lambda p_0^\gamma + (1-\lambda)p_1^\gamma$. If the reverse inequality holds, we say $F$ is $\gamma$-quasi-concave.
\end{definition}

\begin{lemma}\label{lem:convex-psi}
    Let $\psi$ be a function such that $\Psi(x) := \psi(e^x)$ is nondecreasing and convex. Let $\psi(\norm{p}_a^a)$ is $\gamma$-quasi-convex if any of the following condition holds:
    \begin{enumerate}
        \item[(i)] $a \geq \gamma$ and $\psi(\cdot)$ is convex.
        \item[(ii)] $a\gamma \leq 0$.
    \end{enumerate}
    \noindent In contrast, it is $\gamma$-quasi-concave if $a \leq \gamma$ and $\psi$ is concave.
\end{lemma}

\begin{proof}
    Let $P, Q \in \Mfrak_1(\Acal)$ with corresponding densities $p$ and $q$. Pick any $\lambda \in [0, 1]$ and define $r^\gamma = \lambda p^\gamma + \bar{\lambda}q^\gamma$ where $\bar{\lambda} = 1 - \lambda$.

    \noindent\textbf{Convexity condition (i):} When $a \geq \gamma$, $x \mapsto x^{a/\gamma}$ is a convex function. Therefore, we have
    \begin{equation*}
        \norm{r}_a^a = \int (\lambda p^\gamma + \bar{\lambda}q^\gamma)^{a/\gamma}
        \leq \int (\lambda p^a + \bar{\lambda} q^a)
        = \lambda \norm{p}_a^a + \bar{\lambda} \norm{q}_a^a.
    \end{equation*}
    \noindent We now apply $\psi$ to both sides and using the nondecreasing and convexity properties of $\psi$, the desired inequality follows.

    \noindent\textbf{Convexity condition (ii):} By extended Minkowski's inequality (Lemma 1 of~\cite{ghosh2018generalizedentropy}), we obtain
    \begin{align}
                 & \norm{r}_a^\gamma
        \geq \lambda \norm{p}_a^\gamma + \bar{\lambda} \norm{q}_a^\gamma\nonumber                                                                                                                 \\
        \implies & \ln(\norm{r}_a^\gamma) \geq \ln( \lambda \norm{p}_a^\gamma + \bar{\lambda} \norm{q}_a^\gamma )\nonumber                                                                        \\
        \implies & \ln(\norm{r}_a^\gamma) \geq \lambda \ln(\norm{p}_a^\gamma) + \bar{\lambda} \ln(\norm{q}_a^\gamma), \label{eqn:convex-psi-proof-1}                                              \\
        \implies & \frac{a}{\gamma}\ln(\norm{r}_a) \leq \frac{a}{\gamma} \lambda \ln(\norm{p}_a) + \frac{a}{\gamma} \bar{\lambda} \ln(\norm{q}_a)\nonumber                                        \\
        \implies & \Psi\left( \ln(\norm{r}_a^a) \right) \leq \lambda \Psi\left( \ln(\norm{p}_a^a) \right) + \bar{\lambda} \Psi\left( \ln(\norm{q}_a^a) \right).    \label{eqn:convex-psi-proof-2}
    \end{align}
    \noindent All the inequalities follow when both $\gamma(\gamma - a) \geq 0$ and $a/\gamma \leq 0$, or equivalently $\gamma$ and $a$ have opposite signs.

    \noindent\textbf{Concavity condition:} When $a \leq \gamma$, $x \mapsto x^{a/\gamma}$ is a concave function. Therefore, we have
    \begin{equation*}
        \norm{r}_a^a = \int (\lambda p^\gamma + \bar{\lambda} q^\gamma)^{a/\gamma} \geq \lambda \norm{p}_a^a + \bar{\lambda} \norm{q}_a^a.
    \end{equation*}
    \noindent Since $\psi$ is increasing and concave, the result now follows by applying $\psi$ on both sides.
\end{proof}

We now proceed with the proof of Proposition~\ref{thm:gab-entropy-concave}. Let $P, Q \in \Mfrak_{1}(\Acal)$ with corresponding densities $p$ and $q$, and fix any $\lambda \in [0, 1]$. Define the linear combination $r = \lambda p + (1-\lambda)q$. To show that $\eGAB^{(\alpha,\beta),\psi}$ is concave, it is enough to show that the function $\psi(\norm{p}_\alpha^\alpha)/\alpha\beta$ is concave and the function $\psi(\norm{p}_{\alpha+\beta}^{\alpha+\beta})/(\beta(\alpha+\beta))$ is convex. We now carefully treat each possibility of the signs of $\alpha$ and $\beta$, and apply Lemma~\ref{lem:convex-psi} to obtain the required set of conditions. We simply enumerate all the possible cases below, which can be verified by straightforward applications of the previous lemma.
\begin{enumerate}
    \item $\alpha\beta > 0, \beta(\alpha + \beta) > 0, \alpha \leq \gamma, (\alpha +\beta)\gamma \leq 0$ and $\psi$ is concave.
    \item $\alpha\beta > 0, \beta(\alpha + \beta) > 0, \alpha \leq \gamma \leq \alpha + \beta$, and $\psi$ is both convex and concave, i.e., $\psi(x) = x$.
    \item $\alpha\beta > 0, \beta(\alpha+\beta) < 0$, $\max\{ \alpha, \alpha+\beta\} \leq \gamma$ and $\psi$ is concave.
    \item $\alpha\beta < 0, \beta(\alpha + \beta) > 0$, $\min\{\alpha, \alpha + \beta\} \geq \gamma$, and $\psi$ is convex.
    \item $\alpha\beta < 0, \beta(\alpha + \beta) > 0$, $\alpha \geq \gamma$, $(\alpha + \beta)\gamma \leq 0$, and $\psi$ is convex.
    \item $\alpha\beta < 0, \beta(\alpha + \beta) > 0$, $\alpha\gamma \leq 0$, $\alpha + \beta \geq \gamma$, and $\psi$ is convex.
    \item $\alpha\beta < 0, \beta(\alpha + \beta) > 0$, $\alpha\gamma \leq 0$ and $(\alpha+\beta)\gamma \leq 0$. This is impossible.
    \item $\alpha\beta < 0, \beta(\alpha+\beta) < 0$, $\alpha \geq \gamma \geq (\alpha + \beta)$, and $\psi$ is both convex and concave.
    \item $\alpha\beta < 0$, $\beta (\alpha + \beta) < 0$, $\alpha\gamma \leq 0$, $\alpha + \beta \leq \gamma$ and $\psi$ is concave.
\end{enumerate}
\noindent Combining all these cases, we obtain the following summary conditions:
\begin{enumerate}
    \item $\psi(x) = x$, with either $\alpha > 0, \beta \in (-\alpha, 0), \gamma \in [\alpha+\beta, \alpha]$ or $\alpha < 0, \beta \in (0, -\alpha), \gamma \in [\alpha, \alpha+\beta]$.
    \item Convex $\psi$ with $\alpha\beta < 0, \beta(\alpha+\beta) > 0$ and either of (i) $\gamma \in (-\infty, \min\{\alpha,\alpha+\beta\})$, (ii)$\gamma \in [0, \max\{\alpha, \alpha+\beta\}]$, (iii) $\gamma \in (-\infty, \alpha)$ if $\alpha < 0$, (iv) $\gamma \in (-\infty, \alpha +\beta)$ if $\alpha+\beta < 0$.
    \item Concave $\psi$ with either of (i) $\alpha, \beta < 0$ and any $\gamma \geq 0$. (ii) $\alpha(\alpha+\beta) < 0$ and $\gamma \geq \max\{\alpha, \alpha+\beta\}$. (iii) $\alpha < 0, \beta \in (0, -\alpha)$ and any $\gamma \geq 0$.
\end{enumerate}

\subsection{Proof of Lemma~\ref{lem:dtilde-error}}\label{appendix-proof:lem-dtilde-error}

We start with Eq.~\eqref{eqn:dtilde-error}. Note that the difference,
\begin{align}
    \tilde{d}_{\psi}^{(\alpha,\beta)}(p_\epsilon, q)
    - \tilde{d}_{\psi}^{(\alpha,\beta)}((1-\epsilon)^{1/\alpha} p_0, q)
     & = \frac{1}{\alpha\beta}\left[ \psi\left( (1-\epsilon)\inner{p_0,q}_{\alpha,\beta} \right) - \psi\left( \inner{p_\epsilon,q}_{\alpha,\beta} \right) \right]\nonumber \\
     & = \dfrac{\psi'(c)}{\alpha\beta} \epsilon \inner{\delta, q}_{\alpha,\beta},
    \label{eqn:dtilde-error-proof-1}
\end{align}
\noindent where the last line follows from an application of Mean Value Theorem, and $c = \int ((1-\epsilon)p_0^\alpha + \tau \delta^\alpha)q^\beta$ for some $\tau \in [0, \epsilon]$. Since $p_0, q, \delta \in L_{\alpha+\beta}(\mu)$, by an application of Young-type inequalities (see the proof of Lemma~\ref{lem:gab-sufficient} for the technical details), it follows that $c$ is finite as $\int p_0^\alpha q^\beta \dd\mu$ and $\int \delta^\alpha q^\beta \dd\mu$ are finite. As $\psi'$ is continuous, it achieves a finite maximum value over the compact interval with endpoints $\int p_0^\alpha q^\beta \dd\mu$ and $\int p_\epsilon^\alpha q^\beta \dd\mu$. Hence, the right-hand side of~\eqref{eqn:dtilde-error-proof-1} is $O(\epsilon \inner{\delta,q}_{\alpha,\beta})$. This proves~\eqref{eqn:dtilde-error}.

Moving on, let us consider the difference
\begin{align*}
    \tilde{d}_{\psi}^{(\alpha,\beta)}(p_0, q) - \tilde{d}_{\psi}^{(\alpha,\beta)}((1-\epsilon)^{1/\alpha} p_0, q)
     & = \dfrac{1}{\alpha\beta}\left[ \psi\left( (1-\epsilon)\inner{p_0,q}_{\alpha,\beta} \right) - \psi\left( \inner{p_0,q}_{\alpha,\beta} \right) \right] \\
     & = \dfrac{1}{\alpha\beta}\left[ \Psi\left( x_0 + \ln(1-\epsilon) \right) - \Psi(x_0) \right]                                                          \\
     & = \dfrac{1}{\alpha\beta} \Psi'(x_0 + \ln(1-\tau)) \ln(1-\epsilon)                                                                                    \\
     & \leq \dfrac{1}{\alpha\beta} \Psi'\left( \inner{p_0,q}_{\alpha,\beta} \right) \ln(1-\epsilon),
\end{align*}
\noindent where $x_0 = \ln(\inner{p_0,q}_{\alpha,\beta})$, $\tau \in [0, \epsilon]$. In the last line, we use the fact that as $\psi$ is a valid characterizing function for GAB divergence, $\Psi$ is convex and $\Psi'$ is an increasing function.

\subsection{Proof of Theorem~\ref{thm:pythagorean}}\label{appendix-proof:thm-pythagorean}

Note the chain of equalities,
\begin{align}
        & \Delta(p_\epsilon,p_0,q)\nonumber                                                                                                                                                                                                           \\
    ={} & \dGAB^{(\alpha,\beta),\psi}(p_\epsilon, q) - \dGAB^{(\alpha,\beta),\psi}(p_\epsilon, p_0) - \dGAB^{(\alpha,\beta),\psi}(p_0, q)\nonumber                                                                                                    \\
    ={} & -\tilde{d}_{\psi}^{(\alpha,\beta)}(p_\epsilon,p_\epsilon) + \tilde{d}_{\psi}^{(\alpha,\beta)}(p_\epsilon,q) + \tilde{d}_{\psi}^{(\alpha,\beta)}(p_\epsilon,p_\epsilon) - \tilde{d}_{\psi}^{(\alpha,\beta)}(p_\epsilon,p_0) \nonumber        \\
        & \qquad + \tilde{d}_{\psi}^{(\alpha,\beta)}(p_0,p_0) - \tilde{d}_{\psi}^{(\alpha,\beta)}(p_0,q) \label{eqn:pythagorean-proof-1}                                                                                                              \\
    ={} & \tilde{d}_{\psi}^{(\alpha,\beta)}((1-\epsilon)^{1/\alpha} p_0,q) + O(\epsilon \inner{\delta,q}_{\alpha,\beta}) - \tilde{d}_{\psi}^{(\alpha,\beta)}((1-\epsilon)^{1/\alpha}p_0,p_0) + O(\epsilon \inner{\delta,p_0}_{\alpha,\beta})\nonumber \\
        & \qquad  + \tilde{d}_{\psi}^{(\alpha,\beta)}(p_0,p_0) - \tilde{d}_{\psi}^{(\alpha,\beta)}(p_0,q) \label{eqn:pythagorean-proof-2}                                                                                                             \\
    ={} & O(\epsilon v_\delta) + O(\ln(1-\epsilon)) \label{eqn:pythagorean-proof-3}.
\end{align}
\noindent In step~\eqref{eqn:pythagorean-proof-1}, we have made use of identity~\eqref{eqn:dtilde-to-gabdiv}. Steps~\eqref{eqn:pythagorean-proof-2} and~\eqref{eqn:pythagorean-proof-3} follow from a direct application of Lemma~\ref{lem:dtilde-error}.

\subsection{Proof of Theorem~\ref{thm:maxent-dist}}

Let us first define $Q = (q_1, \dots, q_n) \in \Mfrak_{1, n}$ to be the $\alpha$-escort distribution of $P$. Under this reparametrization, we have $q_i \propto p_i^\alpha$, or $p_i \propto q_i^{1/\alpha}$. The constraints given in~\eqref{eqn:maxent-constraints} now translate to
\begin{equation*}
    \sum_{i=1}^n g_r(a_i) q_i = G_r, \ r = 1, \dots, m.
\end{equation*}
\noindent Now we analyze the three cases separately.

\noindent\textbf{Case $\beta(\alpha+\beta) \neq 0$:} In this case, we can form the optimization problem with the Lagrangian parameters $\lambda_1, \dots, \lambda_m$ as to maximize
\begin{equation}
    F(Q) = -\dfrac{1}{\beta(\alpha+\beta)}\psi\left( \frac{S_{(\alpha+\beta)/\alpha}}{S_{1/\alpha}^{\alpha+\beta}} \right) + \dfrac{1}{\alpha\beta}\psi\left(\frac{1}{S_{1/\alpha}^\alpha} \right) + \sum_{r=1}^m \lambda_r \left( \sum_{i=1}^n g_r(a_i)q_i - G_r \right),
    \label{eqn:maxent-dist-proof-1}
\end{equation}
\noindent where $S_a = \sum_{i=1}^n q_i^a = \norm{q}_a^a$ for any $a \geq 0$. Let us also define $T_a = \ln(S_a) = a\ln(\norm{q}_a)$. Differentiating both sides of~\eqref{eqn:maxent-dist-proof-1} with respect to $q_i$, we obtain the first order condition
\begin{equation*}
    \sum_{r=1}^m \lambda_r g_r(a_i)
    = \frac{1}{\alpha\beta}\Psi'\left( T_{\frac{\alpha+\beta}{\alpha}} - (\alpha+\beta)T_{1/\alpha} \right)\left[ \dfrac{q_i^{\beta/\alpha}}{S_{(\alpha+\beta)/\alpha}} - \dfrac{q_i^{1/\alpha-1}}{S_{1/\alpha}} \right]
    + \frac{\Psi'\left( -\alpha T_{1/\alpha} \right)}{\alpha\beta} \dfrac{q_i^{1/\alpha-1}}{S_{1/\alpha}}.
\end{equation*}
\noindent Dividing both sides by $q_i^{1/\alpha-1}$, and using Eq.~\eqref{eqn:c1p-defn}, we obtain
\begin{equation*}
    \alpha\beta\sum_{r=1}^m \lambda_r \frac{g_r(a_i)}{q_i^{1/\alpha-1}}
    = \dfrac{\Psi'(c_1(Q))}{S_{(\alpha+\beta)/\alpha}} q_i^{\frac{\alpha+\beta-1}{\alpha}} + \frac{\Psi'(c_2(Q)) - \Psi'(c_1(Q))}{S_{1/\alpha}}.
\end{equation*}
\noindent In other words,
\begin{equation*}
    q_i \propto \left[ \frac{\Psi'(c_1(Q)) - \Psi'(c_2(Q))}{e^{-c_2(Q)/\alpha}} + \alpha\beta \sum_{r=1}^m \lambda_r \frac{g_r(a_i)}{q_i^{1/\alpha-1}} \right]^{\frac{\alpha}{\alpha+\beta-1}}.
\end{equation*}
\noindent Note that, here we have used the fact that $\Psi$ has to be strictly increasing for the GAB divergence to be a valid statistical divergence measure (see Theorem~\ref{thm:gab-div-nec-suff}) and hence $\Psi'(c_1(Q)) > 0$.

\noindent\textbf{Case $\beta = 0, (\alpha+\beta)\neq 0$:} In this scenario, the objective function with Lagrangian parameters $\lambda_1, \dots, \lambda_m$ turns out to be
\begin{equation*}
    F(Q) = \dfrac{\psi\left( S_{1/\alpha}^{-\alpha} \right)}{\alpha^2} - \dfrac{\psi'\left( S_{1/\alpha}^{-\alpha} \right)}{\alpha} \sum_{i=1}^n \dfrac{q_i}{S_{1/\alpha}^\alpha}\ln\left( \dfrac{q_i^{1/\alpha}}{S_{1/\alpha}} \right) + \sum_{r=1}^m \lambda_r \left( \sum_{i=1}^n g_r(a_i)q_i - G_r \right),
\end{equation*}
\noindent from which, we derive the first-order conditions as
\begin{equation*}
    \sum_{r=1}^m \lambda_r g_r(a_i)
    =  \dfrac{\Psi'(c_2(Q))}{\alpha^2}(1+\ln(q_i)) + \dfrac{\Psi''(c_2(Q))}{\alpha^2 e^{-c_2(Q)/\alpha}} \left[ \sum_{i=1}^n q_i\ln(q_i) + c_2(Q) \right]q_i^{\frac{1}{\alpha}-1}.
\end{equation*}
\noindent Applying our previous characterization results, we have $\Psi'(c_2(Q)) > 0$. However, due to the convexity of $\Psi$, we only know that $\Psi''(c_2(Q)) \geq 0$ without the strict inequality. In fact, when considering the logarithmic Alpha-Beta divergence, we have $\Psi''(\cdot) \equiv 0$. Therefore, through an algebraic rearrangement of the above quantity, we get
\begin{equation*}
    q_i \propto \exp\left[ \dfrac{\alpha^2}{\Psi'(c_2(Q))} \left\{ \sum_{r=1}^m \lambda_r g_r(a_i) + \frac{\Psi''(c_2(Q))}{e^{-c_2(Q)/\alpha}} \left( \sum_{i=1}^n q_i\ln(q_i) + c_2(Q) \right) q_i^{\frac{1}{\alpha}-1} \right\} \right].
\end{equation*}

\noindent\textbf{Case $\beta \neq 0, (\alpha+\beta) = 0$:} In this case, the objective function becomes
\begin{equation*}
    F(Q) = \dfrac{\psi'(1)}{\alpha}\sum_{i=1}^n \ln\left( \dfrac{q_i^{1/\alpha}}{S_{1/\alpha}} \right) - \dfrac{\psi\left( S_{1/\alpha}^{-\alpha} \right)}{\alpha} + \sum_{r=1}^m \lambda_r \left[ \sum_{i=1}^n g_r(a_i)q_i - G_r \right],
\end{equation*}
\noindent where $\lambda_1, \dots, \lambda_r$ are Lagrangian parameters. As in the previous derivations, by differentiating the objective function with respect to $q_i$, we obtain the first-order condition as
\begin{equation*}
    \sum_{r=1}^m \lambda_r g_r(a_i)
    = \dfrac{\Psi'(0)}{\alpha^2} \left[ \dfrac{nq_i^{1/\alpha-1}}{S_{1/\alpha}} - \dfrac{1}{q_i} \right] - \dfrac{\Psi'(c_2(Q))}{\alpha S_{1/\alpha}}q_i^{1/\alpha-1}.
\end{equation*}
\noindent By rearranging the above, noting that $\Psi'(0) > 0$ and rewriting $\lambda_r \mapsto -\alpha^2\lambda_r$, we obtain that
\begin{equation*}
    q_i \propto \left[ \sum_{r=1}^m \lambda_r g_r(a_i) + \dfrac{q_i^{1/\alpha-1}}{e^{-c_2(Q)/\alpha}} \left( n\Psi'(0) - \alpha\Psi'(c_2(Q)) \right) \right]^{-1},
\end{equation*}
\noindent as we wanted to show.

\end{document}